\documentclass[11pt,a4paper]{amsart}
\usepackage{color}
\usepackage{amsfonts}
\usepackage{amssymb}
\usepackage{amsmath,amsxtra}
\usepackage{amsthm}
\usepackage{amscd}
\usepackage{enumitem}
\usepackage{mathrsfs}
\usepackage{stmaryrd}

\usepackage{latexsym}
\usepackage{bm}
\usepackage{graphicx}
\usepackage{wrapfig}
\usepackage{fancybox}
\usepackage[all]{xy}

\usepackage{pgf}
\usepackage{tikz}
\usetikzlibrary{cd}
\newtheorem{theorem}{Theorem}[section]
\newtheorem{lemma}[theorem]{Lemma}
\newtheorem{proposition}[theorem]{Proposition}

\newtheorem{definition}[theorem]{Definition}
\newtheorem{assumption}{Assumption}
\newtheorem{remark}{Remark}

\def\Vect{\operatorname{span}}

\def\dist{\operatorname{dist}}

\def\N{{\mathbf N}}

\def\bl{{\mathcal I}}
\def\I{\bl}

\def\esp#1{\mathbf E\left[#1\right]}

\def\P{\mathbf P}

\def\D{\mathbb D}
\def\F{\mathcal{F}}
\def\R{\mathbf{R}}
\def\car{\mathbf{1}}
\def\maN{\Pois}

\def\dz{\dif z}

\def\ds{\dif s}

\def\KR{\mathfrak K}

\def\Lip{\operatorname{Lip}}

\def\JJ{\mathfrak J}
\def\cont{\mathcal C}

\def\bl{{\mathcal I}}

\def\J{\mathcal J}

\def\emb{\mathfrak e}

\newcommand\diveB{\delta}

\def\H{\mathcal H}
\def\FrF{\mathfrak F}
\def\lec#1{\lesssim_{#1}}
\def\D{\operatorname{d}\!}
\def\dif{\,\D}
\def\Sko{\operatorname{Sko}}
\def\<{\left \langle}
  \def\>{\right\rangle}

\newcommand{\Poist}{\widetilde{\maN}}

\newcommand{\Pois}{\mathcal{N}} 

\def\bX{{\overline{X}_n}}

\def\cyl(X){{\mathbb{X}}}

\begin{document}

\title{Diffusive limits of Lipschitz functionals of Poisson measures}
\author{E. Besan\c con}
\address{LTCI, Telecom Paris, I.P. Paris, France} \email{eustache.besancon@mines-telecom.fr}
\author{L. Coutin}
\address{Institut Math\'ematique de Toulouse, Universit\'e
  P. Sabatier, Toulouse, France}
\email{coutin@math.univ-toulouse.fr}
\author{L. Decreusefond}
\address{LTCI, Telecom Paris, I.P. Paris, France} \email{laurent.decreusefond@mines-telecom.fr}
\author{P. Moyal}
\address{Université de Lorraine, France} \email{pascal.moyal@univ-lorraine.fr}

 \date{}

\begin{abstract}
Continuous Time Markov Chains, Hawkes processes and many other interesting
processes can be described as solution of stochastic differential equations driven
by  Poisson measures. Previous works, using the Stein's method, give the
convergence rate of a sequence of renormalized Poisson measures towards the
Brownian motion in several distances, constructed on the model of the
Kantorovitch-Rubinstein (or Wasserstein-1) distance. We show that many
operations (like time change, convolution) on continuous functions are Lipschitz
continuous to extend these quantified convergences to diffuse limits of Markov
processes and long-time behavior of Hawkes processes.
\end{abstract}
\maketitle{}


\section{Introduction}
Limit theorems for Continuous-Time Markov Chains (CTMC's) have proven to be useful tools 
to approximate the dynamics of the processes under consideration. 
Fluid approximations can be used to determine ergodicity conditions, a first order approximation of the mean dynamics of the process, or to analyze the dynamics of discrete-event systems 
around saturation. Likewise, diffusion approximations (also called, along the various communities, Functional Central Limit Theorems or invariance principles) usually lead to the weak approximation of the (properly scaled) difference between the original CTMC and its fluid limit, to a diffusion process. Such convergence results allow to assess the speed of convergence to 
the fluid limit, and thereby, to gain insights
on the behavior of the considered process when the state space is large, and/or to more easily simulate its paths whenever the dynamics of the original discrete-event CTMC is too intricate. 

The literature regarding weak (fluid or diffusion) approximations of CTMC's is
vast, and so is the range of their fields of applications: for instance, queueing networks
(see e.g. \cite{borovkov1967limit,robert_stochastic_2003,Harrison1987,Kruk2011} 
and references therein), biology and epidemics (e.g. \cite{britton,Etheridge2011,decreusefond2012} and references therein), physics \cite{Chaikin1995}, and so on. From the mathematical standpoint, we can enumerate 
at least three approaches to prove the corresponding convergence theorems. 
The most classical one relies on the so-called Dynkin's Lemma 
(see e.g. Chapter 7 in \cite{ethier86}, or Chapter 7 in \cite{decreusefond2012stochastic}), 
and more 
generally on the semi-martingale decomposition of the considered CTMC, together with the usual weak convergence theorems 
for martingales, to derive the limiting process of the
properly scaled CTMC. 
An alternative way consists in
representing the Markov process as the sum of time-changed Poisson processes and
then to use the well-known limit theorems for such processes, 
see e.g. Chapter 6 in \cite{ethier86}. 
A third, alternative way is to represent the CTMC as the
solution of a stochastic differential equation (SDE) driven by some independent
Poisson random measures (see e.g. \cite{britton}). We
mention these different approaches because not all of them behave nicely 
for what has to be done here.

After fluid and diffusion approximations, the third and next natural step is to evaluate 
the rate of convergence in the diffusion approximation. This is the main 
object of the present work. 
We present hereafter a unified framework, based on the third aforementioned approach, namely, 
on the representation of CTMC's as solution of SDE's driven by Poisson random measures, to 
derive bounds for the convergence in the diffusion approximations of a wide class of CTMC's, under  
various mild conditions on the integrand of these Poisson integrals. By considering a wide range of cases study, from queueing systems to biological models and epidemiological processes, we  show that these assumptions are met by many processes that are prevalent in practice. In many cases, we retrieve existing results concerning the diffusion approximations of the considered processes, and then go one step further, by establishing bounds for the latter convergence. 
We also show that the same procedure can be applied to study the long-run behavior of Hawkes processes. 

For a sequence of processes 
$(X_{n},\,n\ge 1)$ with values in a complete, separable, metric space which converges to a
process $X$, estimating the rate of convergence amounts to computing $\varphi$ such
that
\begin{equation}\label{eq_stability_core4:1}
  \dist_{\mathcal{C}}(\P_{X_{n}},\, \P_{X}):=  \sup_{f\in \mathcal{C}}\esp{f(X_{n})}-\esp{f(X)}\le \varphi(n),
\end{equation}
where $\mathcal{C}$ is the set of test functions. The minimum regularity required 
for the
left-hand-side of \eqref{eq_stability_core4:1} to define a distance is to
suppose $f$ Lipschitz continuous, but not necessarily bounded.

If we take for $\mathcal{C}$ the set of Lipschitz bounded functions, we obtain a
distance which generates the same topology as that of the Prokhorov distance. In
the seventies, many papers (see e.g. \cite{Haeusler1984,Sawyer1972},
and references therein) derived the rates of convergence of functional CLT's such as 
Donsker's theorem, for this metric. They generally obtained
$\varphi(n)=O(n^{-1/4})$ via the Skorohod representation theorem and ad-hoc subtle
computations on the sample-paths themselves.

In the nineties, in his pioneering paper \cite{BarbourSteinmethoddiffusion1990},
Barbour constructed a Malliavin-like apparatus to estimate the rate of
convergence in the Donsker theorem on the Skorohod space $\mathbb D$. The set
$\mathcal{C}$ under consideration is the set of three times Fréchet
differentiable functions on $\mathbb D$ with additional boundedness properties.
Once this functional framework is setup, we can proceed similarly to the 
Stein method in finite dimension (see \cite{kasprzak_diffusion_2017} for a new
application of this approach to the Moran model). 
Let us also mention several recent applications of the Stein method in finite dimension, 
assessing the rate of convergence of the
stationary distributions of various queueing processes: Erlang-A and
Erlang-C systems in \cite{Braverman2016}; a system with reneging and phase-type
service time distributions in \cite{Braverman2017}, and single-server queues in heavy-traffic in \cite{GW19}. 

It is only recently that in
\cite{coutin_steins_2012}, Barbour's result was extended to the convergence in some
fractional Sobolev spaces, instead of $\mathbb D$. This result was then improved in
\cite{Coutin:2019aa} and \cite{coutin:hal-02098892} by allowing, at last,
test functions that are only Lipschitz continuous. When $\mathcal{C}$ is the set of Lipschitz
continuous functions, the induced distance is stronger than the Prokohorov
distance: not only does it imply the convergence in distribution, but also the
convergence of the first moments, see \cite{Villani2003}.

For these test functions, from the theoretical point of view two novelties arise. 
As could be expected, their reduced regularity induces additional 
technicalities but more strikingly, it also yields different rates of
convergence. The benefit is that the applicability is enriched, for this new set of test functions 
embraces many more functions of interest in practice. This can be used, for instance, to derive the
convergence rate for the maximum of a random walk towards the
maximum of the Brownian motion; a result which cannot be established by the
basic Stein method, as we do not have a convenient characterization of the law
of the maximum of the Brownian motion.

Furthermore, the set of Lipschitz functions is remarkably stable with respect to
many operations like time-change, convolution, reflection, etc., so that we can
deduce from a master Theorem, many new convergence rates which do not seem to be
accessible from scratch.

In \cite{besancon_steins_2018}, Stein's method was used to study the rate of convergence 
in the diffusion approximations of the $M/M/1$ and the 
$M/M/\infty$ queues. The two models involved very different ad-hoc techniques
and proved to be difficult to generalize, but led to a satisfying estimate of the
speed of convergence ($n^{-1/2}$, where $n$ is the scaling factor of the
respective models). The approach of the present paper is more general, at the
expense of a lower rate of convergence ($n^{-1/6}$), but covers a much wider class of
processes.

The paper is organized as follows. In Section \ref{sec:preliminaries}, we give
some estimates of the distance between a multivariate point process and its
affine interpolation, depending on the intensity of its jumps. In Section
\ref{sec:lipsch-funct}, we establish that sets of Lipschitz functionals on some function spaces enjoy 
some remarkable stability properties. These properties are crucial to
transfer the convergence rate established in the master
Theorem~\ref{cor:convergence_integral_mesures_poisson}, to more general
processes. In Section
\ref{sec:appliMarkov}, we provide a general result regarding the rate of convergence in the diffusion approximations of a wide class of CTMC's, and then apply this result to various practical processes in queueing, biology and epidemiology. We then quantify the convergence of some functionals of Hawkes processes in Section \ref{sec:limit-theorem-hawkes}.

\section{Preliminaries}
\label{sec:preliminaries}
Throughout the paper, we
fix a time horizon $T>0$. 
For a fixed integer $d$, we denote by $\mathbb D_{T}$ the Skorohod space (i.e. the space of right continuous with left
limits (rcll) functions from $[0,T]$ into $\R^d$). It contains $\cont_{T}$, the space
of continuous functions on $[0,T]$. We denote the sup-norm over $[0,T]$ by
\begin{equation*}
  \|f\|_{\infty,T}=\sup_{t\in[0,T]}\| f(t) \|_{\R^d},
\end{equation*}
for $f\in \mathbb D_{T}$. In what follows, inequalities will be valid up to
irrelevant multiplicative constants, and we write
\begin{equation*}
  a\lec{\alpha} b
\end{equation*}
to mean that there exists $c>0$ which depends only on $\alpha$ such that $a\le
c\, b$.

\subsection{Affine interpolations}
\label{sec:reduct-finite-dimens}

In the forthcoming examples, we have processes whose sample-paths are only
right-continuous-with-left-limits (rcll) and we wish to compare them to the
Brownian motion (BM) or other diffusions whose sample-paths are  continuous.
In the usual proof of the Donsker theorem, the common probability space on
which the convergence is proved is the Skorohod space of rcll functions. Here we aim at a more precise result. 
Actually, our goal is to estimate which factor is
responsible for the slowest rate of convergence: the difference of regularity or
the difference in the dynamics.

It leads us to consider the distance between the affine interpolation of the
processes under scrutiny and the  affine interpolation of the BM, instead of the distance between their
nominal trajectory and that of the BM.

\begin{definition}
  A partition $\pi$ of $[0,T]$ is a sequence
  \begin{equation*}
    \pi=\{0=t_{0}<t_{1}<\ldots<t_{l(\pi)}=T\},
  \end{equation*}
  where $l(\pi)$ is the number of subintervals defined by $\pi$.
  We denote by $|\pi|$ its mesh
  \begin{equation*}
    |\pi|=\sup_{i \in \llbracket 0,l(\pi)-1\rrbracket} |t_{i+1}-t_{i}|.
  \end{equation*}
  We denote by $\Sigma_{T}$, the set of partitions of $[0,T]$.
\end{definition}
For any function $f \in \mathbb D([0,T],\R^{d})$ and any $\pi\in \Sigma_{T}$, we denote by
$\Xi_{\pi}f$ the affine interpolation of $f$ on $[0,T]$ along $\pi$, namely for 
all $t\in [0,T]$, 
\begin{align*}
  (\Xi_\pi f)(t)&=\sum_{i=0}^{l(\pi)-1} \frac{f(t_{i+1})-f(t_{i})}{t_{i+1}-t_{i}}\,(t-t_{i})\car_{[t_{i},t_{i+1})}(t)\\
                &=\sum_{i=0}^{l(\pi)-1} \frac{f(t_{i+1})-f(t_{i})}{\sqrt{t_{i+1}-t_{i}}}\ h_{i}^{\pi}(t),
\end{align*}
where 
\begin{equation*}
  h_{i}^{\pi}(t)=\frac{1}{\sqrt{t_{i+1}-t_{i}}}\int_{0}^{t}\car_{[t_{i},t_{i+1})}(s)\dif s,\quad i\in \llbracket 0,l(\pi)-1\rrbracket. 
\end{equation*}
When $\pi=\left\{iT/n,\, i \in \llbracket 0,n \rrbracket\right\}$, we denote $\Xi_{\pi}$ by
$\Xi_{n}$ and $h_{i}^{\pi}$ by $h_{i}^{n}$ for all $i$. 

\bigskip

When a point process has not too many jumps per subinterval of a
partition~$\pi$, its affine interpolation along $\pi$ does not deviate too much
from its nominal path. More precisely we have the following result, 
  
\begin{theorem} \label{thm_mycore:bound}
  Let $m\in\N^*$, and consider $(X(t),t\in[0,T])$ a $\R^d$-valued point process admitting the representation
  \begin{equation}\label{eq_stability_core:3}
    X(t)=\sum_{k=1}^{m}\left(\int_{0}^{t}\int_{\R^{+}}\car_{\{z\le \varphi_{k}(s,X(s^{-}))\}}\dif \Pois_{k}(s,z)\right) \, \zeta_{k}, 
    \quad t \in [0,T], 
  \end{equation}
  where the $\Pois_{k}$'s are independent Poisson measures of respective intensity measures
  $\rho_k\, \dif s\otimes \dif z$, $k\in\llbracket 1,m \rrbracket$, and for any $k\in \llbracket 1,m \rrbracket$, 
  $\zeta_{k}\in \R^{d}$ and $\varphi_{k}$ is a bounded function $[0,T] \times \R^d\to \R^+$. 
  Then, for any $\pi\in\Sigma_{T}$ we have that 
  \begin{equation*}
    \esp{\|X-\Xi_{\pi}X\|_{\infty,T}}\lec{}\sum_{k=1}^{m}\|\zeta_{k}\|_{\R^{d}}\,\Psi \Bigl(l(\pi),\, \rho_{k}|\pi| \|\varphi_{k}\|_{\infty}\Bigr),
  \end{equation*}
  where for all $(n,x)\in \N^*\times \R$, 
  \begin{equation*}
    \Psi(n,x)=\frac{\log(n e^{x/n})}{\log(n x^{-1}\log(n e^{x/n}) )}\cdotp
  \end{equation*}
\end{theorem}
  %
  %

  \begin{proof}[Proof of Theorem~\protect\ref{thm_mycore:bound}]
  Let $\pi=\left\{t_{i},\,i\in \llbracket 0,n \rrbracket\right\}$.  For any $t\in [0,T]$ there exists $i\leq n-1$ such that
    $t\in[t_{i},\, t_{i+1}) $ and
    \begin{align*}
      \left\|X(t)-\Xi_{\pi}X(t)\right\|_{\R^d}
        &=
        \left\|X(t) - X(t_{i})-\frac{X(t_{i+1})-
            X(t_{i})}{t_{i+1}-t_{i}}\, (t-t_{i})\right\|_{\R^d}\\
        &\le 2
        \sup_{t\in [t_{i},t_{i+1}]}\left\| X(t)- X(t_{i})\right\|_{\R^d},\label{eq_mycore:3}
      \end{align*}
    so that
    \begin{equation*}
      \esp{\parallel X -\Xi_{\pi}X \parallel_{\infty,T}}
      \le 2
      \esp{\max_{i\in \llbracket 0,n-1 \rrbracket} \| X(t_{i+1})-X(t_{i})\|_{\R^d}}.
    \end{equation*}
    Now notice that for all $k\in\llbracket 1,m \rrbracket$, 
    \begin{align*}
      \int_{t_{i}}^{t_{i+1}}\int_{\R^{+}} \car_{\{z\le \varphi_k(s,X(s^{-}))\}}\dif \Pois_k(s,z) &\le  \int_{t_{i}}^{t_{i+1}}\int_{\R^{+}} \car_{\{z\le\|\varphi_k\|_{\infty}\}}\dif \Pois_k(s,z)\\ &:=M_{k,i}^{\pi}.
\end{align*}
But the $M_{k,i}^{\pi},\,i \in \llbracket 0,l(\pi)-1\rrbracket$ are independent Poisson random
variables of respective parameters $\rho_k(t_{i+1}-t_{i})\|\varphi_k\|_{\infty}$, so they are strongly dominated
by a family of $l(\pi)$ independent Poisson random variables of respective parameters~$\rho_k|\pi|\,\|\varphi_k\|_{\infty}$. The result then follows from Proposition~\ref{prop:momentPoisson} below. 
\end{proof}

As to the affine interpolation of the Brownian motion,  according to
\cite[Proposition 13.20]{friz_multidimensional_2010},  there exists $c>0$ such
that for any partition~$\pi$
\begin{equation}\label{eq_stability_core:5}
  \esp{\|\Xi_\pi B-B\|_{W_{\eta,p}}^{p}}^{1/p}\lec{} |\pi|^{1/2-\eta}.
\end{equation}

\subsection{Malliavin gradient}
\label{sec:malliavin-gradient}
We give the minimum elements of Malliavin calculus to understand the sequel.
More advanced material is necessary to prove
Theorem~\ref{thm_stability_core:CD_lipschitz}, see \cite{coutin:hal-02098892}.
We denote by $\H_{T}$ the Hilbert space
\begin{equation*}
  \H_{T}:=\left\{ h\in \cont_{T}, \, \exists \dot h\in L^{2}([0,T],\dif s) \text{ such that } h(t)=\int_{0}^{t}\dot h(s)\dif s \right\}.
\end{equation*}
The function $\dot h$ is unique so that we can define
\begin{equation*}
  \|h\|_{\H_{T}}^{2}=\int_{0}^{T}\dot h(s)^{2}\dif s.
\end{equation*}
By convention, we identify $\H_{T}$ with its dual $(\H_{T})^{*}$.

Consider $(Z_{n},\, n\ge 1)$, a sequence of independent, standard Gaussian random
variables and let $(z_{n},\, n\ge 1)$ be a complete orthonormal basis of
$\H_T$. Then, we know from \cite{ito_convergence_1968} that
\begin{equation}\label{eq_core:6}
\sum_{n=1}^{N} Z_{n}\, z_{n} \xrightarrow{N\to \infty}  B:=\sum_{n=1}^{\infty} Z_{n}\, z_{n}\text{ in }\cont_T \text{ with probability } 1,
\end{equation}
where $B$ is a Brownian motion. We clearly have the diagram
\begin{equation}\label{eq_core:8}
\cont_{T}^{*}\xrightarrow{\emb^{*}} (\H_T)^{*}\simeq \H_T\xrightarrow{\emb} \cont_T,
\end{equation}
where $\emb$ is the canonical embedding from $\H_{T}$ into $\cont_{T}$.
We denote by $\mu$ the law of $B$ on $\cont_{T}$, and by 
$L^{2}(\cont_{T};\mu)$, the space of functions $F$ from $\cont_{T}$ into
$\R$ such that
\begin{equation*}
 \esp{F(B)^{2}}  <\infty.
\end{equation*}
\begin{definition}[Wiener integral]
  The Wiener integral, denoted as $\delta$, is the isometric extension
  of the map
  \begin{align*}
  \delta\, :\,  \emb ^*({\cont_T}^*)\subset\H_{T}&\longrightarrow  L^2(\cont_{T};\mu)\\
  \emb ^*(\eta)  &\longmapsto \<\eta,\, y\>_{{\cont_T}^*,{\cont_T}}.
  \end{align*}
\end{definition}
\noindent This means that if $h=\lim_{n\to \infty } \emb ^{*}(\eta_{n})$ in
$\bl_{1,2}$,
\begin{equation*}
\delta h(B)=\lim_{n\to \infty} \<\eta_{n},\, y\>_{{\cont_T}^*,{\cont_T}} \text{ in }L^{2}(\mu).
\end{equation*}
\begin{definition}
  \label{def_donsker_final:2} Let $X$ be a Banach space. A function $F\, :\, \cont_T\to
  X$ is said to be cylindrical if it is of the form
  \begin{equation*}
  F(y)=\sum_{j=1}^k f_j(\diveB h_1(y),\cdots,\diveB h_k(y))\, x_j,
  \end{equation*}
  where for any $j\in \llbracket 1,k \rrbracket$, $f_j$ belongs to the Schwartz space on $\R^k$, $(h_1,\cdots, h_k)$ are
  elements of $\H_{T}$ and $(x_1,\cdots,x_j)$ belong to $X$. The set of such
  functions is denoted by $\cyl(X)$.

  For $h\in \H_T$,
  \begin{equation*}
  \< \nabla F, \, h\>_{ \H_T}=\sum_{j=1}^k\sum_{l=1}^{k} \partial_lf(\diveB h_1(y),\cdots,\diveB h_k(y))\, \<h_l,\, h\>_{ \H_T}\ x_j,
  \end{equation*}
  which is equivalent to say
  \begin{equation*}
  \nabla F = \sum_{j,l=1}^k \partial_jf(\diveB h_1(y),\cdots,\diveB h_k(y))\, h_l\otimes\ x_j.
\end{equation*}
The space $\mathbb D_{1,2}(X)$ is the closure of cylindrical functions with respect
to the norm of $L^{2}(\cont_T; \H_T\otimes X)$. An element of $\mathbb
D_{1,2}(X)$ is said to be Gross-Sobolev differentiable and $\nabla F$ belongs to
$\H_T\otimes X$ with probability~$1$.

We can iterate the construction to higher order gradients and thus define
$\nabla^{(k)}F$ for any $k\ge 1$, provided that $F$ is sufficiently regular.
\end{definition}
We actually need a stronger notion of weak differentiability. The ordinary
notion of gradient we have just defined, induces that $\nabla F$ belongs almost
surely to $\H_{T}$. Hereafter we also need that it belongs to a smaller space, namely the 
dual of $L^{2 }([0,T],\dif s )$. Recall that we have identified $\H_{T}$ with
itself, so that we cannot identify $L^{2}$ with its dual. It is proved in
\cite{Coutin:2019aa} that
\begin{multline*}
\J_{T}:=\left( L^{2}([0,T],\dif s)\right)^{*}\\
  \simeq \left\{ h\in \cont_{T}, \exists \hat h\in L^{2}([0,T],\dif s ) \text{ such that } h(t)=\int_{0}^{t}\int_{s}^{1}\hat h(u)\dif u \dif s \right\}.
\end{multline*}
\begin{definition}
  We denote by $\Upsilon_{T}$ the subset of functions $F$ in  $\mathbb D_{2,2}(\R)$
which satisfy
  \begin{equation}
    \label{eq_preliminaries:5}
    \left|     \<\nabla^{(2)}F(x)-\nabla^{(2)}F(x+g),\, h\otimes k \>_{\H_{T}} \right|\le \|g\|_{\cont_{T}}\, \|h\|_{L^{2}}\|k\|_{L^{2}},
    \end{equation}
    for any $x\in\cont_{T} $, $g\in \H_{T}$, $h, \,k \in L^{2}([0,T],\dif s)$. This means that
    $\nabla^{(2)}F$ belongs to $\Upsilon_{T}^{\otimes 2}$ and is $\H_{T}$-Lipschitz
    continuous on $\cont_{T}$.
\end{definition}

\section{Lipschitz functionals}
\label{sec:lipsch-funct}
\begin{definition}
Let $(E,d_{E})$ and $(G,d_{G})$ be two metric spaces. A function $F\, :\, E\to G$ is said to be
Lipschitz continuous whenever there exists $c>0$ such that for any $x,y\in E$,
\begin{equation}\label{eq_stability_core:4}
d_{G}(F(x),F(y))\le c \, d_{E}(x,y).
\end{equation}
The minimum value of $c$ such that \eqref{eq_stability_core:4} holds, is the Lipschitz norm of $F$. We denote by
$\Lip_{\alpha}(E\to G,d_{E})$ the set of Lipschitz continuous functions from $E$ to $G$ having Lipschitz norm less
than~$\alpha$.
\end{definition}
When $E$ is a functional space, the set of Lipschitz functions is rich enough to
be stable by some interesting transformations. 
%
%
\begin{lemma}
  \label{thm_mycore:timeChange}
  Let $r$ be a positive integrable function on $[0,T]$ and set
  \begin{equation}\label{eq_stability_core:def_changement_temps}
     \gamma(t)=\int_{0}^{t} r(s)\dif s.
   \end{equation}
   Then, the map
   \begin{align*}
     \Gamma\, :\, \cont_{T}&\longrightarrow \cont_{\gamma(T)}\\
     f&\longmapsto f\circ \gamma
   \end{align*}
   is invertible and Lipschitz continuous with Lipschitz norm~$1$. Moreover, if
   $F$ belongs to $\Upsilon_{T}$ then $F\circ \Gamma$ belongs to $\Upsilon_{\gamma(T)}$.
\end{lemma}
\begin{proof}
  The first part is straightforward. Since $\Gamma $ is a linear continuous and
  bijective map from $\cont_{T}$ to $\cont_{\gamma(T)}$, which maps $\H_{T}$
  (respectively $\J_{T}$)
  bijectively onto $\H_{\gamma(T)}$ (respectively $\J_{\gamma(T)}$), the second
  assertion follows immediately.
\end{proof}

%
%

Let us now fix $m,d\in\N^*$, and 
consider the integral equation 
\begin{equation}
  \label{eq_stability_core:1}
  y(t)=y(0)+\int_{0}^{t} Ay(s)\dif s+ f(t),\quad t\ge 0,
\end{equation}
where $\mathbf A\in\mathfrak M_{d,m}(\R)$ and $f\in\cont_{T}(\R^d)$. 
The unique solution
  of \eqref{eq_stability_core:1} can be written as
   \begin{equation}
    \label{eq_stability_core:2}
    S(f)(t)=Ae^{tA}f(0)-Af(t)+A\int_{0}^{t}e^{(t-r)A}f(r)\dif r.
  \end{equation}
We have the following, 
\begin{theorem}\label{thm_stability_core:ODE}
  For $A\in \mathfrak M_{n}(\R)$, the map
  \begin{align}
    \Theta_{A} \, :\, \cont_{T}(\R^d)&\longmapsto \cont_{T}(\R^d)\nonumber\\
    f&\longmapsto S(f) \label{eq:defTheta}
  \end{align}
  where $S(f)$ is the solution of~\eqref{eq_stability_core:1}, is Lipschitz continuous. Moreover, if
  $F$ belongs to $\Upsilon_{T}$, then so does $F\circ S$. 
\end{theorem}
\begin{proof}
  Since $\Theta_{A}$ is linear, we just have to prove that there exists $c>0$
  such that for any $f\in \cont_{T}$,
  \begin{equation*}
    \|S(f)\|_{\infty,T}\lec{A}  \|f\|_{\infty,T}.
  \end{equation*}
From(\ref{eq_stability_core:2}), we have
  \begin{equation*}
    \|S(f)\|_{\infty,T}\lec{} (\|A\|+\|A\|^{2})e^{T\|A\|}\, \|f\|_{\infty,T}.
  \end{equation*}
  Equation \eqref{eq_stability_core:2} also entails that $S$ is linear and that 
  \begin{equation*}
    S(\H_{T})\subset \H_{T} \text{ and } S(\J_{T})\subset \J_{T},
  \end{equation*}
  hence $\Upsilon_{T}$ is stable by $S$.
The proof is thus complete.
\end{proof}
%
%
\noindent 
Recall (see e.g. Chapter D in \cite{robert_stochastic_2003}), that for any $T>0$ and any $Y\in \mathbb D_T$ such that 
$Y(0)\ge 0$, there exists a unique pair of functions $X_Y$ and $R_Y$ in $\mathbb D_T$ 
such that $X_Y(t)$ is non-negative, $R_Y$ is non-decreasing, $R_Y(0)=0$ and 
for all $t\le T$, 
\[\begin{cases}
X_Y(t) =Y(t)+R_Y(t),\\
\displaystyle\int_0^t X_Y(s)\dif R_Y(t)=0.
\end{cases}\]
Define the mapping 
\begin{align}
    \operatorname{Sko}\, :\, \mathbb D_{T} & \longrightarrow \mathbb D_{T}\label{eq:defSko}\\
    f&\longmapsto \left(s\mapsto f(s)+ \|f^{-}\|_{\infty,s}\right),\notag
  \end{align}
  usually referred to as the {\em Skorokhod reflection map} of $f$. 
Then, it is well known that in the particular case $d=1$, $X_Y$ has the explicit form 
$X_Y=\Sko(Y)$. 
We have the following results,  
\begin{theorem}
\label{thm:Lip}
  The mapping 
  \begin{align*}
    \max\, :\, \mathbb D_{T}&\longrightarrow \mathbb D_{T}\\
    f&\longmapsto (s\mapsto \|f\|_{\infty,s}),
  \end{align*}
the local time map
  \begin{align*}
    \ell^{0}\, :\, \mathbb D_{T} & \longrightarrow \mathbb D_{T}\\
    f&\longmapsto (s\mapsto \|f^{-}\|_{\infty,s})
  \end{align*}
and the Skorohod reflection map $\Sko$  
  are all Lispchitz continuous, and so is 
  for any $\varepsilon>0$ the continuity modulus mapping 
\begin{align*}
  \alpha_{\varepsilon}\, :\, \mathbb D_{T}&\longrightarrow \R^{+}\\
  f&\longmapsto \sup_{|s-s'|\le \varepsilon} \|f(s)-f(s')\|_{\R^d}. 
\end{align*}
\end{theorem}
\begin{proof}
The first three assertions readily follow from the left triangular inequality, entailing 
that for all $f,g$, and all $t\in[0,T]$, 
\begin{equation*}
 \Bigl|\|f\|_{\infty,t}-\|g\|_{\infty,t}\Bigr|\le\|f-g\|_{\infty,t}.
\end{equation*}
Regarding the last assertion, we clearly have for all $s\le T$, 
  \begin{equation*}
    \|f(s)-f(s')\|_{\R^d} \le \|f(s)-g(s)\|_{\R^d}+\|g(s)-g(s')\|_{\R^d}+\|g(s')-f(s')\|_{\R^d}.
  \end{equation*}
Hence for all $\varepsilon$, 
  \begin{equation*}
    \alpha_{\varepsilon}(f)\le 2 \|f-g\|_{\infty,T}+\alpha_{\varepsilon}(g).
  \end{equation*}
  The same holds with the role of $f$ and $g$ permuted, hence
  \begin{equation*}
    \| \alpha_{\varepsilon}(f)- \alpha_{\varepsilon}(g)\|_{\R^d}\le 2 \|f-g\|_{\infty,T},
  \end{equation*}
  and the proof is complete.
\end{proof}
\section{Kantorovitch-Rubinstein distances }
\label{sec:kant-rubinst-dist}

\begin{definition}\label{def_stability_core4:KR}
  For $\mu$ and $\nu$ two probability measures on a metric space $(E,d_{E})$,
  the Kantorovitch-Rubinstein (or Wasserstein-1) distance between $\mu$ and
  $\nu$ is defined as
  \begin{equation*}
    \KR_{E}(\mu,\nu):=\sup_{F\in \Lip_{1}(E\to \R, d_{E})} \int_{E} F\dif \mu -\int_{E} F\dif \nu.
  \end{equation*}
\end{definition}
It is the well known that $(\mu_{n},\, n\ge 1)$ tends to $\nu$ in the
Kantorovitch-Rubinstein topology if and only if $\mu_{n}$ converges in law to
$\nu $ and the sequence of first order moments converges: For some $x_{0}\in E$
and then all $x_{0}\in E$
\begin{equation}\label{eq_stability_core:14}
  \int_{E}d_{E}(x,x_{0})\dif \mu_{n}(x)\xrightarrow{n\to \infty} \int_{E}d_{E}(x,x_{0})\dif \nu(x).
\end{equation}
In \cite{coutin:hal-02098892}, it is proved that
\begin{theorem}
  \label{thm_stability_core:CD_lipschitz}
  Let $(X_{i},\, i\ge 0)$ be a sequence of independent and identically
  distributed random variables belonging to $L^{p}$ for some
  $p\ge 3$. Then,   for any $n\ge 1$,
  \begin{equation}\label{eq_stability_core:borne_Lipschitz}
    \KR_{\cont_{T}}\left( \sum_{i=0}^{n-1}X_{i}h_{i}^{n},\ B \right)\lec{T,\esp{|X_{1}|^{3}}}\ n^{-1/6}\,\log(n).
  \end{equation}
\end{theorem}
To explain the somehow surprising exponent $-1/6$, we quickly describe
the proof of this result. We proceed to a sort of bias-variance decomposition:
for any $N<n$
\begin{align*}
  \KR_{\cont_T}\left( \sum_{i=0}^{n-1}X_{i}h_{i}^{n},\ B \right)&\le \KR_{\cont_T}\left( \sum_{i=0}^{n-1}X_{i}h_{i}^{n},\ \Xi_{N}(\sum_{i=0}^{n-1}X_{i}h_{i}^{n}) \right)\\
  &+ \KR_{\cont_T}\left( \Xi_{N}(\sum_{i=0}^{n-1}X_{i}h_{i}^{n}),\ \Xi_{N}B \right)\\&+ \KR_{\cont_T}\left( \Xi_{N}B,\ B \right)\\&=A_{1}+A_{2}+A_{3}.
\end{align*}
We have seen in \eqref{eq_stability_core:5} that the rightmost term  is bounded
by $n^{-1/2}$. The proof of the latter inequality is based on the scaling invariance and Hölder
continuity of the Brownian motion, hence there are strong reasons to believe
that this rate is optimal. Direct computations show that $A_{1}$ is also bounded
by $n^{-1/2}$. 

It remains to bound the median term~$A_{2}$. For this, we remark that the two processes
under study belong to the finite dimensional space
\begin{equation*}
  \H_{N}=\Vect\left\{h_{i}^{N},i \in \llbracket 0,N-1\rrbracket\right\}.
\end{equation*}
This leads to the definition of finite rank functional,
\begin{definition}
  For $\pi\in \Sigma_{T}$, let
  \begin{equation*}
    \H_{\pi}=\Vect\left\{ h_{i}^{\pi},\, i=0,\cdots,l(\pi)-1 \right\}.
  \end{equation*}
  A function $F\, :\, \cont_T\to \R$ is then said to have a finite rank if there
  exists a partition $\pi\in \Sigma_{T}$ and a function 
  $\varphi_{F}\, :\,
  \H_{\pi}\to \R$ such that
  \begin{equation}\label{eq_stability_core:6}
    F=\varphi_{F}\circ \Xi_{\pi}.
  \end{equation}
  It amounts to saying that $F$ depends only on the value of $x$ by its value at
  the points of $\pi$. For any $\pi\in \Sigma_{T}$, we denote by $\F^{\pi}$ the set of functions of finite rank associated to 
  $\pi$, that is, such that \eqref{eq_stability_core:6} holds.
\end{definition}
\noindent A straightforward adaptation of the proof of the main theorem in
\cite{coutin:hal-02098892} yields
\begin{theorem}
  \label{thm_stability_core:Stein_finiterank}
  For any $\pi\in \Sigma_{T}$, set 
  \begin{equation*}
    \Lip^{\pi}_{1}=\Lip_{1}(\cont_T\to \R, \|\ \|_{\infty,T})\cap \F^{\pi}.
  \end{equation*}
  Then, we have for all $n\in\N^*$, 
  \begin{multline}\label{eq_stability_core:borne_rangfini}
    \sup_{F\in \Lip^{\pi}_{1}} \left\{\esp{F\left(  \Xi_{\pi}\left(\sum_{i=0}^{n-1}X_{i}h_{i}^{n}\right) \right)}-\esp{ F\left( \Xi_{\pi}B \right)}\right\}\\\lec{T,\esp{|X_{1}|^{3}}}  \frac{|\pi|^{-1}}{\sqrt{n}}\, \log(n).
  \end{multline}
\end{theorem}
If we apply this theorem to $\Xi_{N}$, we are in a position similar (but not equivalent) to that of bounding the Wasserstein-1
distance between a sum of independent (but not identically distributed) random
vectors in dimension~$N$, and the standard Gaussian distribution in $\R^{N}$. 
In this finite dimensional situation, the best known
results \cite{Raic2018,MR4168389} show that the multiplying constant of the factor
$n^{-1/2}\log(n)$ depends linearly on the dimension, a fact that we retrieve here.

The exponent $-1/6$ is then obtained by choosing the optimal $N$ as a function
of $n$. We have strong
confidence that each of  the two steps gives the optimal rate, and consequently, that this overall rate is optimal.
For many practical applications, test functions in $ \Lip^{\pi}_{1}$ are sufficient, see e.g. 
the simple and practical functionals addressed in Section \ref{sec:lipsch-funct}. 
This leads us to introduce the following distance, 
\begin{equation*}
  \FrF_{\cont_T}^{\pi}(\mu,\nu)=\sup_{F\in \Lip^{\pi}_{1}} \left\{\int_{\cont_T}F\dif \mu-\int_{\cont_T}F\dif \nu\right\},
\end{equation*}
for $\mu,\nu$ two probability measures.

  If we allow to take the supremum over  a smaller set of test functions like three times Fréchet
  differentiable with bounded derivatives (as in
  \cite{BarbourSteinmethoddiffusion1990,kasprzak_diffusion_2017}), we get a
  convergence rate bounded by $n^{-1/2}$. However, we can get this rate even for much
  less regular test functions. Let
  \begin{equation*}
    \Lip_{1}^{\Upsilon}=\Lip_{1}(\cont_{T}\to \R, \|\ \|_{\infty,T})\cap \Upsilon_{T}.
  \end{equation*}
Then, setting  for all $\mu,\nu$, 
  \begin{equation*}
    \JJ_{\cont_{T}}(\mu,\nu)=\sup_{F\in \Lip^{\Upsilon}_{1}} \left\{\int_{\cont_T}F\dif \mu-\int_{\cont_T}F\dif \nu\right\}, 
  \end{equation*}
 it is shown in \cite{Coutin:2019aa} that  
  \begin{equation*}
    \JJ_{\cont_{T}}\left( \Xi_{\pi}\left(\sum_{i=0}^{n-1}X_{i}h_{i}^{n}\right),\ B \right)\lec{T,\pi} \frac{1}{\sqrt{n}}\cdotp
  \end{equation*}
In what follows, we deal with renormalized stochastic integrals with respect to Poisson
measures as in \eqref{eq_stability_core:3}. One of the trick of our proof is to
replace the random integrands by  deterministic and supposedly close functions.
This technical step introduces an error which converges to zero at rate
$n^{-1/4}$ (see \eqref{eq:bofbof}), so much slower than $n^{-1/2}\log(n)$. We thus  redefine the
usual rates at which the convergence hold.
  \begin{definition}
  A sequence of rcll processes $(X_{n},\, n\ge 1)$ is said to converge in
  distribution  \textsl{
    at the usual rates } to a process $Z$ if for any $\pi\in \Sigma_{T}$ and
  any~$n\ge 1$, 
  \begin{equation}
    \label{eq_stability_core:9}
    \tag{\dag}
    \begin{cases}
      \KR_{\cont_{T}}(\Xi_{n}X_{n},\, Z)&\lec{T} n^{-1/6}\log(n),\\
      \FrF_{\cont_{T}}^{\pi}(\Xi_{n}X_{n},\, Z)& \lec{T}|\pi|^{-1}n^{-1/4},\\
      \JJ_{\cont_{T}}(\Xi_{n}X_{n},\, Z)&\lec{T} n^{-1/4}.
    \end{cases}
  \end{equation}
\end{definition}
\noindent A straightforward application of the previous results leads to the following, 
%
%

\begin{theorem}
  \label{thm_stability_core:convergence_Poisson}
  Let $P_{n}$ be a Poisson process on $[0,T]$ of intensity $n$. Then, the
  sequence of  processes $\left(\overline P_{n},\, n\ge 1\right)$ defined by 
  \begin{equation}\label{eq_stability_core:7}
    \overline P_{n}(t):=\frac{P_{n}(t)-nt}{\sqrt{n}}
  \end{equation}
converges at the usual rates to a Brownian motion. Furthermore,
  \begin{equation*}
    \esp{\|\overline P_{n}-\Xi_{n}\overline P_{n}\|_{\infty,T}}\lec{T} \frac{\log(n)}{\sqrt{n}}\cdotp
    \end{equation*}
\end{theorem}
    \begin{proof} Fix $n\ge 1$. 
      According to Theorem~\ref{thm_mycore:bound}, we get that 
      \begin{equation*}
        \esp{\|\overline P_{n}-\Xi_{n}\overline P_{n}\|_{\infty,T}}\lec{T} \frac{\Psi(n,1)}{\sqrt{n}}\lec{T} \frac{\log(n)}{\sqrt{n}}\cdotp
      \end{equation*}
      Furthermore,
      \begin{align*}
        \Xi_{n}\overline P_{n}&=\sum_{i=0}^{n-1} \frac{P_{n}((i+1)/n)-P_{n}(i/n)- 1}{\sqrt{n}\sqrt{1/n}}\, h_{i}^{n}\\
        &\stackrel{\text{dist.}}{=} \sum_{i=0}^{n-1} (X_{i}-1)\,  h_{i}^{n}
        \end{align*}
        where $(X_{k},\, k\ge 1)$ is an IID sequence of Poisson random variables
        of parameter~$1$. The result then follows from \eqref{eq_stability_core:borne_Lipschitz} and \eqref{eq_stability_core:borne_rangfini}.
    \end{proof}
We are now in a position to state our main result. It consists of an 
extension of the last Theorem to stochastic integral with respect to Poisson
measures, 
\begin{theorem}
  \label{cor:convergence_integral_mesures_poisson}
  Let $r$ be a positive integrable function on $[0,T]$ and $\gamma$ defined
  by \eqref{eq_stability_core:def_changement_temps}. Let $\Pois^{n}$ be a Poisson
  measure on $[0,T]\times \R^{+}$ of intensity measure $n\, \dif t\otimes \dif  z$. 
  Consider the compensated Poisson measure
  \begin{equation*}
    \dif  \Poist^{n}(t,z)=\dif \Pois^{n}(t,z)-  n\, \dif t\otimes \dif z.
  \end{equation*}
Define also the process $\overline R_{n}$ by
  \begin{equation*}
    \overline R_{n}(t)=\frac{1}{\sqrt{n}}\,\int_{0}^{t}\int_{\R^{+}}\car_{\{z\le r\left(s\right)\}}  \dif \Poist^{n}(s,z),\, t \in [0,T]. 
  \end{equation*}
  Then, $(\overline R_{n},\, n\ge 1)$ converges at the usual rates to $B\circ \gamma$. 
\end{theorem}
\begin{proof}
  According to \cite[Theorem 16]{bremaud_point_1981}, for any $n \ge 1$,
  $\overline R_{n}\circ \gamma^{-1}$ has the distribution of
   $\overline P_{n}$ defined in \eqref{eq_stability_core:7}. 
Therefore, from Theorem \ref{thm_stability_core:convergence_Poisson}, 
$(\overline R_{n}\circ \gamma^{-1},\, n\ge 1)$ converges at the usual rates to
    $B$. The result then follows from Lemma~\ref{thm_mycore:timeChange}.
\end{proof}

\section{Application to Continuous time Markov chains}
\label{sec:appliMarkov}

\subsection{General settings}
\label{sec:cont-time-mark}

Numerous Continuous time Markov chains (CTMC's) can be described as solutions of stochastic differential equations
with respect to a finite family of Poisson measures. 
For any $m\in\N^*$, any family $\left(\zeta_{1}, ...,\zeta_{m}\right)_{1\leq k\leq m}$ of elements of $\R^d$ and
any array $\left(\rho_k\right)_{1\leq k\leq m}$ of mappings from  $[0,T]\times\mathbb D_T$ to $\R$, consider the $\R^d$-valued process $X$ defined as the solution
of the SDE
$$X(t)=
X(0)+\sum_{k=1}^{m}\left(\int_0^t\int_{\R^+}\car_{\left\{z\leq\rho_k\left(s,X\right)\right\}}\dif\Pois_k(s,z)\right)\zeta_k,\quad t\le T,$$
where $X(0) \in \R^d$ is fixed, and $\left(\Pois_{k}\right)_{1\leq k\leq m}$ denote $m$ independent Poisson measures of unit intensity
$\dif s\otimes\!\!\dif z$.

Fix $n\in\N^*$ and an array $\mathbf\alpha:=(\alpha_1,...,\alpha_m)\in (\R)^m$.
We scale the process $X$ by replacing for all $k$, the measure  $\Pois_k$ by a Poisson measure $\Pois^{\alpha_k}_k$ of intensity
$(n^{\alpha_k}\dif s)\otimes\!\!\dif z$, and normalizing in space by~$n$. Then the process $\bX:=n^{-1}X_{n}$ is  the solution of the following SDE: 
For
any $t\ge 0$,
\begin{equation}
\label{eq:defbarXn}
\bX(t)=\bX(0)+\frac1n \left(\sum_{k=1}^{m} \int_0^{t}\int_{\R^+}\car_{\left\{z\leq\rho_k\left(s, n\bX\right)\right\}}\dif\Pois^{\alpha_k}_k(s,z)\right)\,\zeta_k.
\end{equation}
The key assumption on our scaling is the following:
\begin{assumption}[{\bf Law of large numbers scaling}]
  \label{hypo:LFGN}
  For any $k\in\llbracket 1,m \rrbracket$, there exists a mapping $r_k\in
  \Lip_{c}(\R^{d}\to \R,\|\,\|_{\R^{d}})$ such that for all $n\in\N^*$, all $x\in\R^d$ and $t\le T$,
  $$n^{\alpha_k-1}\rho_k(t,n\,x) = r_k(t,x),$$
  and such that
   \begin{equation}
   \label{eq:LFGN2}
   \sup_{n\ge 1}\esp{\sup_{t\le T} \left|r_k(t,\overline{X}_{n}(t)\right|}\le K_k.
   \end{equation}
\end{assumption}
Observe that crucially, under Assumption \ref{hypo:LFGN} the $r_k$'s do not depend on $n$.
We denote for any $k$, by $\Poist^{\alpha_k}_k$, the compensated Poisson measures of $\Pois^{\alpha_k}_k$, that is,
we let
\begin{equation*}
 \dif\Poist^{\alpha_k}_k(s,z) = \dif\Pois^{\alpha_k}_k(s,z) - n^{\alpha_k} (\dif s \otimes\! \dif z).
\end{equation*}
Then, for all $n$ and all $k$, the $\R^d$-valued process
$M_{n,k,\bX}$, defined for all $t\le T$ by
\begin{equation}\label{eq:defMbX}
M_{n,k,\bX}(t) = \left(\int_0^{t}\int_{\R^+}\car_{\left\{z\leq n^{1-\alpha_k}\ r_k\left(s,\bX(s^-)\right)\right\}}\dif\Poist^{\alpha_k}_k(s,z)\right)\, \zeta_k
\end{equation}
is a martingale with respect to the natural filtration of the Poisson measures. 
It then follows from (\ref{eq:defbarXn}) that for all $n$ and $t$,
\begin{equation}
\label{eq:EDS}
\bX(t)=\bX(0)+\sum_{k=1}^{m} \left(\int_0^{t}r_k\left(s,\bX(s)\right)\dif s\right)\,\zeta_k\, +\sum_{k=1}^{m}n^{-1} M_{n,k,\bX}(t).
\end{equation}
In view of Assumption ~\ref{hypo:LFGN} and the Cauchy-Lipschitz Theorem, there exists a unique solution $\Lambda$ in $\mathcal C([0,T];\, \R^{d})$
to the integral equation
\begin{equation}
\label{eq:defLambda}
\Lambda(t)=\Lambda(0)+\sum_{k=1}^{m}\left( \int_0^tr_k(s,\Lambda(s))\dif s\right)\,\zeta_k,\quad t\ge 0.
\end{equation}
We have the following law of large numbers, 
\begin{theorem}
  \label{lemma:scaling}
Assume that there exists a  solution $\Lambda$ of (\ref{eq:defLambda}) on $[0,T]$.
  Suppose that Assumption~\ref{hypo:LFGN} holds. Then, there exists $c>0$ such
  that 
  \[\sup_{n\ge 1}\esp{\sup_{t\le T} \|\overline{X}_{n}(t)\|}\le e^{cT}\]
  and for all $n\ge 1$, 
   \begin{equation}
   \esp{\sup_{t\leq T}\left\|\bX(t)-\Lambda(t)\right\|}\le \left(\esp{\left\|\bX(0)-\Lambda(0)\right\|}+Kn^{-1/2}\right)e^{cT}.\label{eq_mycore:1}
  \end{equation}
\end{theorem}
\begin{proof}
By denoting for all $k\in\llbracket 1,m \rrbracket$, by $\lambda_k$ the Lipschitz constant of the 
mapping $r_k$, we readily get that for all $t\ge 0$, 
\begin{multline*}
\left\|\bX(t)-\Lambda(t)\right\| \le \left\|\bX(0)-\Lambda(0)\right\| \\
+ \left(\sum_{k=1}^{m}\lambda_k\|\zeta_k\|\right)\int_0^{t}\left\|\bX(s)-\Lambda(s)\right\|\dif s+\sum_{k=1}^{m}\left\|n^{-1}M_{n,k,\bX}(t)\right\|,
\end{multline*}
and it is then a classical consequence of Gronwall Lemma that for 
$c:=\sum_{k=1}^{m}\lambda_k\|\zeta_k\|$, 
\begin{multline*}
\esp{\sup_{t\leq T}\left\|\bX(t)-\Lambda(t)\right\|}\\
 \le 
\left(\esp{\left\|\bX(0)-\Lambda(0)\right\|}+\sup_{t\leq T}\sum_{k=1}^{m}\left\|n^{-1}M_{n,k,\bX}(t)\right\|\right)e^{cT}.
\end{multline*}
We conclude using the Burkholder-Davis-Gundy inequality in view of (\ref{eq:LFGN2}). 
\end{proof}
We now turn to the so-called diffusion scaling of the process $X$. We  study the
sequence of processes $(U_n,\, n\ge 1)$ defined for all $n$ by
$$U_n=n^{1/2}\ \Bigl(\bX-\Lambda\Bigr),\quad n\ge 1.$$
In view of~\eqref{eq:defMbX}, we see that the integrands in the $M_{n,k,\bX}$'s, $k\in\llbracket 1,m \rrbracket$, are random processes that depend on~$n$. In this form, we cannot directly use Corollary~\ref{cor:convergence_integral_mesures_poisson}. The
key idea is to introduce  intermediate martingales with deterministic
integrands. The error made with this additional process is easily controlled by
sample-paths estimates.

So let us introduce for all $n$, the $\R^d$-valued martingales $(M_{n,k,\bX})_{k\in\llbracket 1,m \rrbracket}$,  $(M_{n,k,\Lambda})_{k\in\llbracket 1,m \rrbracket}$ and
$\overline{M}_{n,\Lambda}$, respectively defined for all $t\le T$ by
\[\begin{cases}
  M_{n,k,\Lambda}(t)& = \displaystyle\left(\int_0^{t}\int_{\R^+}\car_{\left\{z\leq n^{1-\alpha_k}\,r_k(s,\Lambda(s))\right\}}\dif\Poist^{\alpha_k}_k(s,z)\right)\,\zeta_k\,;
  \\
  \overline{M}_{n,\bX}(t)& = \displaystyle n^{-1/2}\, \sum_{k=1}^m M_{n,k,\bX}(t)\,;\nonumber
  \\ 
  \overline{M}_{n,\Lambda}(t) &=\displaystyle n^{-1/2}\, \sum_{k=1}^m M_{n,k,\Lambda}(t)\,. \nonumber
\end{cases}\]
Observe that the crucial difference between $\overline{M}_{n,\bX}$ and
  $\overline{M}_{n,\Lambda}$ is that the indicator functions appearing in
  the latter involve deterministic processes and thus  $\overline{M}_{n,\Lambda}$ has independent increments. 
From \eqref{eq:EDS} and \eqref{eq:defLambda}, for all $n$, for any $t\in [0,T]$, we have that
\begin{multline}
\label{eq:defU}
U_n(t)=U_n(0)\\+n^{1/2}\, \sum_{k=1}^{m}\left(
\int_0^{t}\big(r_k(s,\bX(s))-r_k(s,\Lambda(s))\big)\dif
s\right)\, \zeta_k
+\overline{M}_{n,\bX}(t).
\end{multline}
To obtain a formal functional central limit theorem for the sequence
$(U_n,\, n\ge 1)$, we make the following additional assumptions.

\begin{assumption}[{\bf Diffusion scaling}] 
\textit{ }
  \label{hypo:FCLT}
  \begin{itemize}
  \item[(i)] The initial conditions are such that 
  \begin{equation}
    \label{eq:controlCI}
    \esp{\left\| \bX(0)-\Lambda(0)\right\|}\lec{}  
   n^{-1/2},
    \end{equation}
  \item[(ii)] For all $n$, all $k\in\llbracket 1,m \rrbracket$ and all $t\in [0,T]$,
    \begin{equation}
    \label{eq:decomprho}
    {n^{1/2}}\left(r_k(t,\bX(t))-r_k(t,\Lambda(t))\right)=\Bigl(\langle L_k, U_n(t)\rangle_{\R^d}+E_{n,k}(t) \Bigr),
    \end{equation}
    where for all $k\in\llbracket 1,m \rrbracket$, $L_{k}\in \R^{d}$ and
    \begin{equation*}
    \label{eq:controleps}
   \esp{\left\|E_{n,k}\right\|_{\infty,T}}\xrightarrow{n\to \infty} 0.
    \end{equation*}
    \end{itemize} 
 \end{assumption}
\begin{theorem}
\label{thm:main}
    Assume that Assumptions \ref{hypo:LFGN} and \ref{hypo:FCLT} hold. Then 
$(U_{n},\, n\ge 1)$ converges at the usual rates to
    \begin{equation*}
      \Theta_{\mathbf A}\left( \sum_{k=1}^{m}(B_{k}\circ \gamma_{k})\,\zeta_{k} \right),
    \end{equation*}
where $\Theta_{\mathbf A}$ is the map associated to the matrix $\mathbf A=(L_1,...,L_m) \otimes (\zeta_1,...,\zeta_m)$ by (\ref{eq:defTheta}), 
$(B_{k})_{k\in \llbracket 1,m\rrbracket}$ are independent one dimensional Brownian
motions, and for any $k\in\llbracket 1,m \rrbracket$, $\gamma_{k}$ is defined in function of the mapping $t\mapsto r_k(t,\Lambda(t))$ through  \eqref{eq_stability_core:def_changement_temps}. 
     \end{theorem}

  \begin{proof}
In view of \eqref{eq:defU} and (\ref{eq:decomprho}), we have for all $n\ge 1$,  
\begin{equation}
  \label{eq_stability_core:10}
  U_{n}=\Theta_{\mathbf A}\left(\sum_{k=1}^{m}\left(\int_0^t E_{n,k}(s)\dif s\right)\,\zeta_k 
  +\overline{M}_{n,\overline{X}_{n}} \right).
\end{equation}
But first, according to Assumption \ref{hypo:FCLT} we get that 
\begin{equation}
\label{eq:bof}
  \esp{\left\| \sum_{k=1}^{m}\left(\int_0^.E_{n,k}(s)\dif s\right)\,\zeta_k\right\|_{\infty,T}} \xrightarrow[]{n\to \infty}0.
\end{equation}
Second, from the Burkholder Davis Gundy inequality, for any $k$ there exists a constant $c_k$ such that
  \begin{multline*}
    \esp{\left\|n^{-1/2}M_{n,k,\bX}-n^{-1/2}M_{n,k,\Lambda}\right\|_{\infty,T}}\\
    \begin{aligned}
      &\le c_k\ n^{-1/2}
       \Vert\zeta_k\Vert\ n^{1/2}\esp{\int_0^{T}\left|r_k(t,\bX(t))-r_k(t,\Lambda(t))\right|\dif
        t}^{1/2}\\
      &\le c_{k}\  \Vert\zeta_k\Vert\ \esp{\int_0^{T} \|\bX(t)-\Lambda(t)\|_{\R^d}\dif t}^{1/2},
    \end{aligned}
  \end{multline*}
 so it follows from \eqref{eq_mycore:1} together with \eqref{eq:controlCI}, that 
 \begin{equation}
 \label{eq:bofbof}
    \esp{\left\|n^{-1/2}M_{n,k,\bX}-n^{-1/2}M_{n,k,\Lambda}\right\|_{\infty,T}}\lec{}  
   n^{-1/4}.
  \end{equation}
On another hand, or any $k\in \llbracket 1,m\rrbracket$, $M_{n,k,\Lambda}$ has the distribution of
    \begin{equation*}
      t\longmapsto \left(\int_{0}^{t}\int_{\R^{+}} \car_{\{z\le r_{k}\left(s,\Lambda(s)\right)\}}\dif \Poist^{n}(s,z)\right)\, \zeta_{k},
    \end{equation*}
    where $\Pois^n$ is a Poisson measure of intensity $n \dif t\otimes \dif z$,  
and it follows from Corollary~\ref{cor:convergence_integral_mesures_poisson} that the process 
$n^{-1/2}M_{n,k,\Lambda}$ converges at the usual rates to $\left(B_k\circ\gamma_k\right)\,\zeta_k$.  
The result then follows from the Lipschitz continuity of 
$\Theta_{\mathbf A}$ in view of the representation~\eqref{eq_stability_core:10} and Theorem~\ref{thm_stability_core:ODE}, together with (\ref{eq:bof}) and (\ref{eq:bofbof}). 
  \end{proof}

\subsection{The Telegraph process}
\label{sec:onoff}
Let $Y_i,\,i\in \llbracket 1,n \rrbracket$ be an IID family of CTMC's taking values in $\{0,1\}$, 
with transition intensity $\sigma_0$ (resp., $\sigma_1$) from state $0$ to state $1$ (resp., from state $1$ to state $0$), and let $X_n$ be the process defined by 
$$X_n(t)=\sum_{i=1}^{n}Y_i(t),\quad t\ge 0.$$
The processes $Y_i,\,i\in \llbracket 1,n \rrbracket$ are often called {\em telegraph processes}, and model various phenomena in finance, physics, networking and biology. The states $0$ and $1$ can be respectively interpreted as `Off' and `On' modes, and so $X_n$ counts the number of telegraph processes in `On' mode at each time. (It is often itself called, a telegraph process.) 
Denote by $\pi_0$ and $\pi_1$, the common stationary probability of the $Y_i$'s, 
$i\in \llbracket 1,n \rrbracket$, namely 
$$\pi_0=\frac{\sigma_1}{\sigma_0+\sigma_1}\quad\text{ and }\quad\pi_1=\frac{\sigma_0}{\sigma_0+\sigma_1}\cdot$$
It is immediate to observe that for any $n$, the limiting distribution $X_n(\infty)$ of the process $X_n$ 
is binomial of parameters $n,\,\pi_1$, and that 
$$\sqrt{n}\left({X_n(\infty) \over n} - \pi_1\right)
\Longrightarrow \mathcal N(0,\pi_0\pi_1).$$
At the process level, it is shown in \cite{Decreusefondlang} 
that under suitable assumptions on the initial conditions, for all $T>0$, 
\[U_n=\sqrt{n}\left(\frac{X_n}{n}-\Lambda\right) \Longrightarrow \Theta(Y) \quad \mbox{ in }\mathbb D([0,T],\R),\]
where 
\begin{itemize}
\item For a fixed $\Lambda(0)\in\R+$, 
\begin{equation}
\label{eq:defLambdaONOFF}
\Lambda(t)=\pi_1+(\Lambda(0)-\pi_1)\exp(-(\sigma_1+\sigma_0)t),\quad t\ge 0;
\end{equation}
\item 
For any $f$, $\Theta(f)$  denotes the unique solution in  $\mathbb D([0,T],\R)$ of the equation 
\begin{equation*}
\label{eq:equadiffONOFF}
{g(t)=(\sigma_1-\sigma_0)\int_0^tg(s)\dif s+f(t),\quad t\ge 0};
\end{equation*} 
\item The process $Y$ is defined by 
$$Y(t)=\int_0^t\sqrt{\sigma_0(1-\Lambda(s))}\dif B_1(s)-\int_0^t\sqrt{\sigma_1\Lambda(s)}\dif B_2(s),\quad t\ge 0,$$
for $B_1$ and $B_2$, two independent standard Brownian motions. 
\end{itemize}

\noindent We have the following result, 
\begin{proposition}
\label{prop:ONOFF}
Suppose that condition (i) of Assumption \ref{hypo:FCLT} is satisfied. 
Then the convergence of $(U_{n},\, n\ge 1)$ to $\Theta(Y)$ occurs at the usual rates. 
\end{proposition}

\begin{proof}
Let for all $n\ge 1$, $\overline X_n(t)=X_n(t)/n$, $t\ge 0$. Then by the very definition of $X_n$ we have the equality in distribution  
\begin{multline*}
\overline X_n(t) \overset{(d)}{=}\overline X_n(0)+\int_0^t\int_{\mathbb{R}^+}\mathbf{1}_{\{z\leq n\sigma_0(1-\overline X_n(s^-))\}}\dif\Pois_1(s,z)\\
-\int_0^t\int_{\mathbb{R}^+}\mathbf{1}_{\{z\leq n\sigma_1 \overline X_n(s^-)\}}\dif\Pois_{2}(s,z),\quad t\ge 0, 
\end{multline*}
where $\Pois_{1}$ and $\Pois_{2}$ denote two independent Poisson random measures of common intensity $\dif s \otimes \!\dif z$ on $(\R_+)^2$, 
representing the overall ``up'' and ``down'' jumps, respectively. So $\overline X_n$ satisfies the SDE (\ref{eq:defbarXn}) for $d=1$, $m=2$, $\alpha_1=\alpha_2=0$, $\zeta_1=1$, $\zeta_2=-1$, and for the mappings 
$\rho_1: (t,y) \mapsto \sigma_0(n-y)$ and $\rho_2: (t,y) \mapsto \sigma_1y$. 
It is then clear that Assumption \ref{hypo:LFGN} is satisfied for 
\[r_1:(t,x) \longmapsto \sigma_0(1-x)\quad \mbox{ and }\quad r_2:(t,x)\longmapsto -\sigma_1x,\]
which are obviously Lipschitz continuous with respect to their second variable. 
Plainly, $\Lambda$ defined by (\ref{eq:defLambdaONOFF}) is the unique solution of (\ref{eq:defLambda}) in the present case. 
As $r_1$ and $r_{2}$ are linear in their second variable, condition (ii) in Assumption \ref{hypo:FCLT} is clearly satisfied for 
$L_1=-\sigma_0$, $L_2 = \sigma_1$ and $E_{n,1}\equiv E_{n,2}\equiv 0$ 
for all $n$, and so the corresponding operator $\Theta_{\mathbf A}=\Theta$ defined by (\ref{eq:defTheta}) 
is linear, continuous and therefore Lipschitz continuous. 
We conclude using Theorem \ref{thm:main}.
\end{proof}

\subsection{Two classical queueing systems}
It is a simple matter to show that the single-server queue $M/M/1$ and the infinite server queue $M/M/\infty$, 
for which the speed of convergence in the functional central limit theorem was already addressed in \cite{besancon_steins_2018}, can be studied in the present framework. The convergence rate obtained hereafter is $n^{1/6}$, 
slightly slower than that obtained in  \cite{besancon_steins_2018} ($n^{1/2}$),
where we used representation methods that are specific to each of these models.

\subsubsection{The infinite server queue}
We consider an M/M/$\infty$ queue: a potentially unlimited number of servers
attend customers that enter the system following a Poisson process of intensity
$\lambda>0$, requesting service times that are exponentially distributed of
parameter $\mu>0$. 
Let for all $t\ge 0$, $L^\#(t)$ denote the number of customers in the system at time $t$. 
It is well known that $L^\#$ is an ergodic Markov process with stationary distribution Poisson of parameter 
$\lambda/\mu$. 
Let us scale this process in space and time, by dividing the number of customers by $n$, while multiplying 
the intensity of arrivals by $n$. Namely, we denote for all $t\ge 0$, $\overline L^\#_n(t)=L^\#(t)/n$. 
Then, it is a classical result (see e.g. \cite{borovkov1967limit}, Theorem 6.14 in \cite{robert_stochastic_2003}, or a measure-valued extension of the result in \cite{decreusefond_functional_2008}), that under suitable assumptions on the initial conditions, for all $T>0$, 
\[U_n=\sqrt{n}\left(\overline L^\#_n-\Lambda\right) \Longrightarrow \Theta(Y) \quad \mbox{ in }\mathbb D([0,T],\R),\]
where 
\begin{itemize}
\item For a fixed $\Lambda(0)\in\R+$, 
\begin{equation}
\label{eq:defLambdainfinite}
\frac{\lambda}{\mu}-\left(\Lambda(0)-\frac{\lambda}{\mu}\right)\exp(-\mu t),\quad t\ge 0;
\end{equation}
\item 
For any stochastic process $f$, $\Theta(f)$  denotes the unique solution in $\mathbb D([0,T],\R)$ of the SDE
\begin{equation*}
\label{eq:equadiffinfinite}
{g(t)=-\mu\int_0^tg(s)\dif s+f(t),\quad t\ge 0};
\end{equation*} 
\item The process $Y$ is defined by 
$$Y(t)=\sqrt{\lambda} B_1(t) - \int_0^t\sqrt{\mu\Lambda(s)}\dif B_2(s),\quad t\ge 0,$$
for $B_1$ and $B_2$, two independent standard Brownian motions. 
\end{itemize}

\noindent We have the following result, 
\begin{proposition}
\label{thm:infinite}
Suppose that condition (i) of Assumption \ref{hypo:FCLT} is satisfied. 
Then the convergence of $(U_{n},\, n\ge 1)$ to $\Theta(Y)$ occurs at the usual rates. 
\end{proposition}

\begin{proof}
Then, for all $n\ge 1$ it is easily checked that the resulting scaled process 
$\overline L^\#_n$ satisfies the SDE 
\begin{multline*}
\overline L^\#_n(t)=\overline L^\#_n(0)+\frac{1}{n}\int_{0}^{t}\int_{\mathbb{R}^+}\mathbf{1}_{\{z\leq \lambda}\}\dif \Pois^1_1(s,z)
\\-\frac{1}{n}\int_{0}^{t}\int_{\mathbb{R}^+}\mathbf{1}_{\{z\leq \mu \overline L_n^\#(s^-)\}}\dif \Pois^1_2(s,z),\quad t\ge 0,\end{multline*}
where  $\Pois^1_1$ and $\Pois^1_2$ denote two independent random Poisson measures of intensity 
$n\dif s \otimes \dif z$ on $(\mathbb{R_+})^2$. 
This precisely means that the process $\overline L^\#_n$ satisfies (\ref{eq:defbarXn}) for $d=1$, $m=2$, 
$\alpha_1=\alpha_2=1$, $\zeta_1=1$, $\zeta_2=-1$, and for the mappings 
$\rho_1: (t,y) \mapsto \lambda$ and $\rho_2: (t,y) \mapsto {\mu y\over n}$.
It is then clear that Assumption \ref{hypo:LFGN} is satisfied for 
$r_1:(t,y) \longmapsto \lambda$ and $r_2:(t,y)\longmapsto \mu y.$
Clearly, $\Lambda$ defined by (\ref{eq:defLambdainfinite}) is the unique solution of the differential equation 
(\ref{eq:defLambda}) in the present context. Then, plainly, 
condition (ii) in Assumption \ref{hypo:FCLT} is satisfied for 
$L_1=0$, $L_2 = \mu$ and $E_{n,1}\equiv E_{n,2}\equiv 0$ 
for all $n$, and we get that $\Theta_{\mathbf A}=\Theta$ in (\ref{eq:defTheta}). 
We conclude again using Theorem \ref{thm:main}.
\end{proof}

\subsubsection{The single server queue}
In this section we consider a M/M/$1$ queue: a single-server attends without vacations, customers that enter following a Poisson process of intensity $\lambda>0$, and the service times are IID with exponential distribution of parameter $\mu>0$.   
It is then immediate that the process $\left(L^{\dag}(t),\,t\ge 0\right)$ counting the
number of customers in the system at all times, is a birth and death process, 
that is ergodic if and only if $\lambda/\mu<1$.
This process can be represented as  follows: for all $t\ge 0$, 
\begin{multline*}
\begin{aligned}
 L^{\dag}(t)&=x+\int_0^t\int_{\mathbb{R}^+}\mathbf{1}_{\{z\leq\lambda\}}\dif\Pois_1(s,z)
 -\int_{0}^{t}\int_{\mathbb{R}^+}\mathbf{1}_{\{z\leq\mu\}}\mathbf{1}_{\{L^{\dag}(s^-)>0\}}\dif\Pois_2(s,z)\\
 &=x+\int_0^t\int_{\mathbb{R}^+}\mathbf{1}_{\{z\leq\lambda\}}\dif\Pois_1(s,z)-\int_{0}^{t}\int_{\mathbb{R}^+}\mathbf{1}_{\{z\leq\mu\}}\dif\Pois_2(s,z)
 \end{aligned}\\
 +\int_{0}^{t}\int_{\mathbb{R}^+}\mathbf{1}_{\{z\leq\mu\}}\mathbf{1}_{\{L^{\dag}(s^-)=0\}}\dif\Pois_2(s,z),
\end{multline*}
where $x$ is the number-in-system at time 0, and $\Pois_1$, $\Pois_2$ and $\Pois_3$ 
stand again for three independent random Poisson measures of intensity $\dif s \otimes \dif z$ on $(\mathbb{R_+})^2$. 
The standard Law-of-Large-Numbers scaling of this process is performed by 
multiplying the arrival and service intensities by a factor $n$, while increasing the number of customers in the initial state by the same multiplicative factor, and dividing the number of customers in the system at any time 
by $n$: equivalently, for all $t\ge 0$ we set 
\begin{multline}\label{eq:SkoMM1}
\overline{L}^\dag_n(t) = x+{1\over n}\int_0^t\int_{\mathbb{R}^+}\mathbf{1}_{\left\{z\leq\lambda\right\}}\dif\Pois^1_1(s,z)
 -{1\over n}\int_{0}^{t}\int_{\mathbb{R}^+}\mathbf{1}_{\left\{z\leq\mu\right\}}\dif\Pois^1_2(s,z)\\
 +{1\over n}\int_{0}^{t}\int_{\mathbb{R}^+}\mathbf{1}_{\left\{z\leq\mu\right\}}\mathbf{1}_{\left\{\overline L^{\dag}_n(s^-)=0\right\}}\dif\Pois^1_2(s,z),
\end{multline} 
where $\Pois^1_1$ and $\Pois^1_2$ denote two 
Poisson random measures of intensity $n \ds \otimes \dz$. 
The following result completes the classical diffusion limit presented e.g. in Proposition 5.16 in \cite{robert_stochastic_2003}, and Theorem 1 in \cite{besancon_steins_2018}, 

\begin{proposition}
Let 
\begin{equation}
\label{eq:defLambdasingle}
\Lambda(t)=\left(x+(\lambda-\mu)t\right)^+, \quad t\ge 0,
\end{equation}
 and $Z$ be the process defined by 
\begin{equation*}
Z(t) =\sqrt{\lambda + \mu}B(t),\quad t\ge 0,
 \end{equation*}
 for $B$ a standard Brownian motion. 
Then the following convergence holds at the usual rate, 
\[U_n=\sqrt{n}\left(\overline L^\dag_n-\Lambda\right) \Longrightarrow \Sko(Z) \quad \mbox{ in }\mathbb D([0,T],\R),\]
where the mapping $\Sko$ is defined by (\ref{eq:defSko}). 
\end{proposition}
\begin{proof}
The relation (\ref{eq:SkoMM1}) means that 
$\overline L^\dag_n = \Sko(\bX)$ for the process $\bX$ define for all $t\ge 0$ by 
\begin{equation*}
\bX(t) = x+{1\over n}\int_0^t\int_{\mathbb{R}^+}\mathbf{1}_{\{z\leq\lambda\}}\dif\Pois^1_1(s,z) 
           -{1\over n}\int_{0}^{t}\int_{\mathbb{R}^+}\mathbf{1}_{\{z\leq\mu\}}\dif\Pois^1_2(s,z). \end{equation*}
It is then immediate that both Theorems \ref{lemma:scaling} and \ref{thm:main} can be applied to 
the sequence of processes $\{\bX\}$ for $d=1$, $m=2$, $\alpha_1=\alpha_2=1$, $\zeta_1=1$, $\zeta_2=-1$ and for all $(t,x)$, $r_1(t,x)=\lambda$ and $r_2(t,x)=\mu$. 
We obtain (for $\Theta_A:=\textbf{Id}$), that the convergence 
\[V_n:=\sqrt{n}\left(\overline X_n-\Gamma\right) \Longrightarrow Z \quad \mbox{ in }\mathbb D([0,T],\R)\]
occurs at the usual rate, where for all $t\ge 0$, we define 
\begin{equation*}
\Gamma(t) =x+(\lambda-\mu)t,
\end{equation*}
and observe the equality in distribution 
\[Z(t) =\sqrt{\lambda} B_1(t) - \sqrt{\mu}B_2(t),\]
for $B_1$ and $B_2$, two independent standard Brownian motions. It is immediate that $\Lambda=\Sko(\Gamma)$, so the result follows from the Lipschitz continuity of the mapping $\Sko$, proven in Theorem \ref{thm:Lip}.  
\end{proof}

\subsection{SIR epidemics}
We consider a population of constant size $n$, in which individuals can go
through three states, susceptible, infectious ad then removed. The duration of
any infection follows an exponential distribution of parameter $\gamma$ and for
each couple infectious/susceptible, a contagion occurs from the former at an
exponential rate $\lambda$. All the involved r.v.'s are assumed independent. At
any time $t\ge 0$, we let $S^n(t), I^n(t)$ and $R^n(t)$ denote respectively the
number of Susceptible, Infected and Recovered individuals, and let
$X^n(t)=\begin{pmatrix}S^n(t)\\I^n(t)\end{pmatrix}$. The processes are scaled by
defining for all $t\ge 0$, $\overline{S}_n(t) =S^n(t)/n$ ,
$\overline{I}_n(t)=I^n(t)/n$ and $\overline{R}^n(t)=R^n(t)/n$, which represent
respectively the proportions of Susceptible, Infected and Recovered individuals
in the whole population at time $t$. We also let
$\overline{X}^n(t)=\begin{pmatrix}\overline{S}^n(t)\\\overline{I}^n(t)\end{pmatrix}$,
$t\ge 0$. 
A large-graph limit and a functional central limit theorem for the process $\overline{X}$ are given in Chapter 2 of \cite{britton}, together with similar results regarding the related SEIR, SIRS and SIS models. Observe that a hydrodynamic limit for a SIR process propagating on a heterogeneous population (meaning that a - non necessarily complete - graph connects susceptible to infectious individuals) is provided in \cite{decreusefond2012}, completing the result in \cite{volz2008sir}. The following result makes precise the speed of convergence in the functional CLT for the complete-graph case, given in \cite{britton}. We are confident that similar results hold for the other related models addressed in Chapter 2 of \cite{britton}, however we only consider here the SIR case for brevity, 
\begin{proposition}
  Let $\Lambda:t\longmapsto \begin{pmatrix}s(t)\\i(t)\end{pmatrix}$ be the
  unique solution of the system of ODE's
  \begin{equation}
    \label{eq:defLambdaSIR}
    \begin{cases}
      s'(t)  &=-\lambda s(t)i(t)\\
      i'(t) &=\lambda s(t)i(t)-\gamma i(t)
    \end{cases}, \quad t \ge 0.\end{equation} Let $\Theta(Y)$ denote the unique
  solution in $\mathbb D([0,T],\R)$ of the following SDE of unknown
  $g=\begin{pmatrix}g_1\\g_2\end{pmatrix}$,
  \begin{multline*}
    g(t)=\lambda\int_0^t \Bigl(i(u)g_1(u)+s(u)g_2(u)\Bigl)\dif u \begin{pmatrix} -1\\1\end{pmatrix}\\
    +\gamma\int_0^tg_2(u)\dif u \begin{pmatrix}
      0\\-1\end{pmatrix}+Y(t),\quad t\ge 0,
  \end{multline*}
  where for all $t\ge 0$,
$$Y(t)=\int_0^t\sqrt{\lambda s(u)i(u)} \dif B_1(u) \begin{pmatrix} -1\\1\end{pmatrix} + \int_0^t\sqrt{\gamma i(u)} \dif B_2(u) \begin{pmatrix} 0\\-1\end{pmatrix},$$
for $B_1$ and $B_2$, two independent standard Brownian motions. Then, if assertion (i) of
Assumption \ref{hypo:FCLT} is satisfied, the following convergence holds at the
usual rates,
\[U_n=\sqrt{n}\left(\overline X_n-\Lambda\right) \Longrightarrow \Theta(Y) \quad
  \mbox{ in }\mathbb D([0,T],\R).\]
\end{proposition}

\begin{proof}
  By the very definition of the SIR dynamics, for any $n$ the process $\overline
  X^n$ admits the following representation: for all $t\ge 0$,
  \begin{multline*}
    \bX(t)=\bX(0)+\frac{1}{n}\int_0^t\int_{\mathbb{R}^+}\mathbf{1}_{\{z\leq\lambda
      n\overline S^n(u^-)\overline
      I^n(u^-)\}}\dif\Pois_1(u,z) \begin{pmatrix}
      -1\\1\end{pmatrix}\\+\frac{1}{n}\int_0^t\int_{\mathbb{R}^+}\mathbf{1}_{\{z\leq\gamma
      n\overline I^n(u^-)\}}\dif\Pois_2(u,z)\begin{pmatrix}
      0\\-1\end{pmatrix},
  \end{multline*}
  for $\Pois_1$ and $\Pois_2$, two Poisson random measures of unit intensity, so
  we fall again into the settings of Section \ref{sec:cont-time-mark} for $d=2$,
  $m=2$, $\alpha_1=\alpha_2=0$, $\zeta_1=\binom{-1}{1}$ and
  $\zeta_2=\binom{0}{-1}$. It is then clear that Assumption \ref{hypo:LFGN} is
  satisfied for the mappings
  \begin{align*}
    r_1 \,:\,\R_+\times [0,1]^2 &\longrightarrow \R,\\
    \left(t,\binom{y_1}{y_2}\right) &\longmapsto \lambda y_1y_2;\\
    r_2 \,:\,\R_+\times [0,1]^2 &\longrightarrow \R,\\
    \left(t,\binom{y_1}{y_2}\right) &\longmapsto \gamma y_2.
  \end{align*}
  Indeed, to check that $r_1$ is Lipschitz continuous on its second variable,
  just observe that for all $y=\binom{y_1}{y_2}\in [0,1]^2$,
  $y'=\binom{y_1}{y_2}\in [0,1]^2$ for all $t$,
  \begin{equation*}
    \left|r_1(t,y') - r_1(t,y)\right| =\lambda \left|y'_2(y'_1-y_1)+y_1(y'_2-y_2)\right| 
    \le \lambda\left\|y' - y\right\|,\end{equation*}
  so Theorem \ref{lemma:scaling} is satisfied. 
  Now, define again for all $t$ and $n$, $U^n(t)=n^{1/2}\left(\bX(t)-\Lambda(t)\right)$. 
  To check that assumptions \ref{hypo:FCLT} holds, let us also observe that for all $y,y',t$ as above, 
  we also have that 
  \begin{equation*}
    r_1(t,y') - r_1(t,y) =\lambda \left(y_2(y'_1-y_1)+y_1(y'_2-y_2)+(y'_2-y_2)(y'_1-y_1)\right) 
    ,\end{equation*}
  entailing that for all $n$ and $t$, 
  \begin{equation*}
    n^{1/2}\left(r_1(t,\bX(t))-r_1(t,\Lambda(t)) \right)
    = \langle L_1(t),U_n(t)\rangle_{\R^2}+E_{1,n}(t),
  \end{equation*}
  for
  \[L_1(t) = \lambda \begin{pmatrix}s(t)\\i(t)\end{pmatrix}\quad \mbox{ and
    }\quad E_{1,n}(t)=\left(\overline S^n(t)-s(t)\right)\left(\overline
      I^n(t)-i(t)\right).\] Then, from (\ref{eq_mycore:1}) we obtain that
  \[\esp{\parallel E_{1,n} \parallel_{T,\infty}} \lec{} n^{-1/2}.\]
  Likewise, it is immediate that $r_2$ is Lipschitz continuous in its second
  variable, and we get for all $n$ and $t$ that
  \begin{equation*}
    n^{1/2}\left(r_2(t,\bX(t))-r_2(t,\Lambda(t)) \right)
    = \langle L_2(t),U_n(t)\rangle_{\R^2} \quad \mbox{for }L_2(t):=\binom{0}{\gamma},
  \end{equation*}
  so assumptions \ref{hypo:FCLT} hold, and we apply again Theorem
  \ref{thm:main}.
\end{proof}

\subsection{The Moran model}
\label{subsec:moran}

In this section we consider a biological model known as the Moran model : in a population of size $n$, each individual bears a gene liable to take two forms : $A$ and $B$. Each individual has one single parent and its child inherits the genetic form of its parent. To each couple of individuals is associated an exponential clock of unit rate, and each time the clock of a given couple rings, 
one element of the couple, drawn uniformly at random, dies, while the other one gives birth to another individual bearing the same gene. In addition, every gene of type $A$ mutes independently to type $B$ at rate $\nu_1$ and every gene of type $B$ mutes independently to type $A$ at rate $\nu_2$.
For all $t\ge 0$, we let $X_n(t)$ denote the number of individuals bearing gene $A$ in the population at time $t$. 
The process $X_n$ is scaled by dividing by $n$ the exponential rates, together with the number of individuals, so that 
$\overline{X}_n(t)=X_n(t)/n$ represents the proportion of individuals carrying the gene of type A at time $t$. A functional Stein method is applied to this model in \cite{kasprzak_diffusion_2017}. 
The following result is based on an alternative representation of the process $X_n$, 

\begin{proposition}
	Let $\Lambda$ be the solution in $\mathbb C\left([0,T]\right)$ 
	of the integral equation 
	\begin{equation*}
		\Lambda(t)=\Lambda(0)+\int_0^t(\nu_2-(\nu_1+\nu_2)\Lambda(s))\dif s,\quad t\ge 0. 
	\end{equation*}
	with $\Lambda(0)$ such that (i) of assumption \ref{hypo:FCLT} is satisfied. 
	Then the following convergence holds at the usual rate, 
	\[U_n=\sqrt{n}\left(\overline X_n-\Lambda\right) \Longrightarrow \Theta(Y) \quad \mbox{ in }\mathbb D([0,T],\R),\]
	where for all process $f$, $\Theta(f)$ is the only solution of the SDE 
	\[y(t)=y(0)+(\nu_1+\nu_2)\int_0^ty(s)\dif s + f(t),\quad t\ge 0,\] 
	and where 
	$$Y(t)=\int_0^t\sqrt{2\Lambda(s)(1-\Lambda(s))} \dif B(s),\quad t\ge 0,$$ 
	for $B$ a standard Brownian motion.
\end{proposition}

\begin{proof}
	For all $n$, at each given time $s$ there are $X_n(s)(n-X_n(s))$ couples gathering an '$A$-individual' and a '$B$-individual'. 
	Thus for all $t\ge 0$, $X_n(t)$ can be represented as 
	\begin{multline*}
		X_n(t)=X_n(0)+\int_0^t\int_{\mathbb{R}}\mathbf{1}_{\{z\leq X_n(s)(n-X_n(s))\}}\dif\Pois_1(s,z)\\
		-\int_0^t\int_{\mathbb{R}}\mathbf{1}_{\{z\leq X_n(s)(n-X_n(s))\}}\dif\Pois_2(s,z)\\
		+\int_0^t\int_{\mathbb{R}}\mathbf{1}_{\{z\leq \nu_2(n-X_n(s))\}}\dif\Pois_3(s,z)\\
		-\int_0^t\int_{\mathbb{R}}\mathbf{1}_{\{z\leq \nu_1X_n(s)\}}\dif\Pois_4(s,z),
	\end{multline*}
	where $\Pois_i,i\in\llbracket 1,4 \rrbracket$, denote four independent Poisson random measures of unit intensity. 
	Consequently we get that 
	\begin{multline*}
		\overline{X}_n(t)=\overline{X}_n(0)+\frac{1}{n}\int_0^{t}\int_{\mathbb{R}}\mathbf{1}_{\{z\leq n\overline{X}_n(s)(n-n\overline{X}_n(s))\}}\dif\Pois^{-1}_1(s,z)\\-\frac{1}{n}\int_0^{t}\int_{\mathbb{R}}\mathbf{1}_{\{z\leq n\overline{X}_n(s)(n-n\overline{X}_n(s))\}}\dif\Pois^{-1}_2(s,z)\\+\frac{1}{n}\int_0^{t}\int_{\mathbb{R}}\mathbf{1}_{\{z\leq \nu_2(n-n\overline{X}_n(s))\}}\dif\Pois^{-1}_3(s,z)\\-\frac{1}{n}\int_0^{t}\int_{\mathbb{R}}\mathbf{1}_{\{z\leq \nu_1n\overline{X}_n(s))\}}\dif\Pois^{-1}_4(s,z),
	\end{multline*}
	where $\Pois^{-1}_i, i\in\llbracket 1,4 \rrbracket$, are four independent Poisson measures of intensity $n^{-1}\dif s\otimes \dif z$. 
	We fall back into the settings of Section \ref{sec:cont-time-mark} for $d=1$, $m=4$, $\zeta_1=\zeta_3=1$, $\zeta_2=\zeta_4=-1$ and 
	$\alpha_1=...=\alpha_4=-1$.  
	Assumption \ref{hypo:LFGN} holds for 
	$$r_1(t,x)=r_2(t,x)=x(1-x),\quad t\ge 0,\quad x\in[0,1],$$ 
	which are Lipschitz continuous in their second variable, as 
	for all $x,y\in [0,1]$ and all 
	$t\ge 0$ we get
	\begin{equation*}
		\left|x(1-x) - y(1-y)\right| \le 3 \left| x - y \right|. 
	\end{equation*}
	Assumption \ref{hypo:LFGN} also obviously holds for
	$$r_3(t,x)=\nu_2(1-x) \text{ and } r_4(t,x)=\nu_1x\quad t\ge 0,\quad x\in[0,1].$$ 
	On another hand, an immediate computation shows that for all $x,y\in [0,1]$, 
	\[y(1-y)-x(1-x) = (y-x)(1-2x)-(y-x)^2,\]
	so (\ref{eq:decomprho}) holds for 
	\[\begin{cases}
		L_1(t) =L_2(t)=(1-2\Lambda(t)),\quad t\ge 0,\\
		L_3(t)=\nu_2, L_4(t)=-\nu_1, \quad t\ge 0,\\
		E_{n,1}(t) =E_{n,2}(t)=\sqrt{n}(\bX(t)-\Lambda(t))^2,\quad t\ge 0,\\
		E_{n,3}(t) =E_{n,4}(t)=0,\quad t\ge 0,
	\end{cases}\]
	and from (\ref{eq_mycore:1}) we obtain that 
	\[\esp{\parallel E_{i,n} \parallel_{T,\infty}} \lec{} n^{-1/2},\quad i\in\{1,2\}.\]
	So Assumption \ref{hypo:FCLT} holds, and we apply again Theorem \ref{thm:main}. 
\end{proof}

\section{Limit theorem for Hawkes processes}
\label{sec:limit-theorem-hawkes}
In this section we turn to the cases of Hawkes processes. 
We let $\varphi\, :\, \R^{+} \to \R^{+}$ be an integrable function, such that
\begin{equation*}
\kappa:=  \int_{0}^{\infty}\varphi(t)\dif t<1 \text{ and } \int_{0}^{\infty} t^{1/2}\varphi(t)\dif t<\infty.
\end{equation*}
We denote by $\Phi$ the first primitive of $\varphi$:
\begin{equation*}
  \Phi(t)=\int_{0}^{t}\varphi(s)\dif s,\quad t\ge 0.
\end{equation*}
We also need to consider the iterated convolution products of $\varphi$ with
itself:
\begin{equation*}
  \varphi^{(1)}=\varphi \text{ and } \varphi^{(k)}=\varphi*\varphi^{(k-1)}.
\end{equation*}
The function
\begin{equation*}
  \psi=\sum_{k=1}^{\infty} \varphi^{(k)} 
\end{equation*}
plays an important role in the representation of the process to be defined hereafter. Note that
\begin{equation*}
  \int_{0}^{\infty} \psi(t)\dif t=\sum_{k\ge 1}\kappa^{k}=\frac{\kappa}{1-\kappa}\cdotp
\end{equation*}
For $\mu>0$, according to \cite{hawkes74,massoulie96}, there exists a point process $N$ (unique in distribution) such that $N$ admits the compensator
\begin{equation*}
  t\longmapsto \mu t+\int_{0}^{t} \varphi(t-s)\dif N(s).
\end{equation*}
In \cite{BACRY20132475}, it is proved that
\begin{equation}
  \label{eq_stability_core:11}
  \esp{\sup_{v\in [0,1]} \left|\frac1{n}N(nv) -\rho v\right|^{2}}\lec{} \frac{1}{n},
\end{equation}
where $\rho=(1-\kappa)^{-1}\mu$. Furthermore, it is shown that
\begin{equation*}
  \sqrt{n}\left( \frac1{n}N(nv) -\rho v \right)\xrightarrow[n\to \infty]{\text{dist. in }\mathbb D} B\left( \frac{\rho}{\sqrt{1-\kappa}}\, v \right).
\end{equation*}
Our goal is to assess the rate of this convergence. To that end, we use a particular
construction of $N$ based on a Poisson measure $M $ of intensity measure $\dif s\otimes
\dif z$. For all $t\ge 0$, we know from \cite{massoulie96} and references therein, 
that we can write 
\begin{equation*}
  N(t)=\int_{0}^{t}\int_{\R^{+}} \car_{\{z\le \mu+\int_{0}^{t}\varphi(t-s)\dif N(s)\}}\dif M(s,z).
\end{equation*}
Denote also 
\begin{equation*}
  W(t)=\int_{0}^{t}\int_{\R^{+}} \car_{\{z\le \mu+\int_{0}^{t}\varphi(t-s)\dif N(s)\}}\left( \dif M(s,z) - \dif s\dif z\right),
\end{equation*}
the corresponding compensated integral, so that $W$ is a local martingale with respect to the
filtration induced by $M$:
\begin{equation*}
  \F_{t}=\sigma\left\{M([0,s]\times A),\, 0\le s\le t, \, A\in\mathcal{B}(\R^{+})\right\},\,t\ge 0. 
\end{equation*}
For a process $Z$, we denote by
\begin{equation*}
  Z^{(n)}(t)=Z(nt),\  \overline{Z}^{(n)}(t)=\frac{1}{\sqrt{n}}Z^{(n)}(t),\ \tilde{Z}^{(n)}(t)=\frac{1}{n}Z^{(n)}(t),\quad t\ge 0. 
\end{equation*}
From \cite[Chapter 13]{JacodCalculstochastiqueproblemes1979}, we also know that
\begin{equation*}
  W^{(n)}(t)=\int_{0}^{nt}\int_{\R^{+}} \car_{\{z\le \mu+n\int_{0}^{t}\varphi(nt-ns)\dif \tilde{N}^{(n)}(s)\}}\left( \dif M(s,z) - \dif s\dif z\right),\, t\ge 0. 
\end{equation*}
These considerations result in the following lemma, 
\begin{lemma}
  \label{lem:ConvergenceW}
  For all $n \ge 1$, we have 
  \begin{equation*}
    \KR\left( \Xi_{n}(\overline{W}^{(n)}),\Xi_{n}(B\circ \gamma) \right)\lec{} n^{-1/6}\log (n),
  \end{equation*}
  where
  \begin{math}
    \gamma(t)=\rho t,\quad t\ge 0. 
  \end{math}
\end{lemma}
\begin{proof}
  Fix $n\ge 1$. As $\tilde{N}^{(n)}$ is asymptotically close to $\gamma$ 
   we consider the process 
  \begin{equation*}
    R^{(n)}:\,t \longmapsto \int_{0}^{nt}\int_{\R^{+}} \car_{\{z\le \mu+n\int_{0}^{t}\varphi(nt-ns)\rho \dif s\}}\left( \dif M(s,z) - \dif s\dif z\right). 
  \end{equation*}
  It has the same distribution as
  \begin{equation*}
    \hat R^{(n)}:\,t \longmapsto \int_{0}^{t} \int_{\R^{+}} \car_{\{z\le \mu+\int_{0}^{t}\varphi(t-s)\rho \dif s\}}\left( \dif \Pois^1(s,z) - n\dif s\dif z\right),
  \end{equation*}
  where $\Pois^1$ is a Poisson measure of intensity $n\, \dif s\otimes \dif z$.
  According to Corollary~\ref{cor:convergence_integral_mesures_poisson},
  \begin{equation*}
    \KR\left( \Xi_{n}(\frac{1}{\sqrt{n}}\hat R^{(n)}),\,\Xi_{n}(B\circ \gamma) \right)\lec{} n^{-1/6}\log(n),
  \end{equation*}
  which is equivalent to
  \begin{equation}
    \label{eq_stability_core:13}
    \KR\left( \Xi_{n}(\overline{R}^{(n)}),\,\Xi_{n}(B\circ \gamma) \right)\lec{} n^{-1/6}\log(n).
  \end{equation}
We now want to estimate the error made by considering $\overline{R}_{n}$ instead
of $\overline{W}_{n}$. 
  Let for all $t\ge 0$, 
  \begin{align*}
    r(t)&=\mu+n\int_{0}^{t}\varphi(nt-ns)\dif \tilde{N}^{(n)}(s)\\
    r'(t)&= \mu+n\int_{0}^{t}\varphi(nt-ns)\rho \dif s.
  \end{align*}
  We have
  \begin{multline*}
    \|\Xi_{n}(\overline{R}^{(n)})-\Xi_{n}(\overline{W}^{(n)})\|_{\infty,1}\\
    \begin{aligned}
      &\lec{}\sup_{i\in \llbracket 0,n-1\rrbracket}\left|(\overline{R}^{(n)}-\overline{W}^{(n)})(\frac{i+1}{n})-(\overline{R}^{(n)}-\overline{W}^{(n)})(\frac{i+1}{n})\right|\\
      &=\sup_{i\in \llbracket 0,n-1\rrbracket}\left| \int_{i/n}^{(i+1)/n}\int_{\R^{+}} \left( \car_{\{r\le r(s)\}}-\car_{\{z\le r'(s)\}} \right)\dif \Poist(s,z) \right|.      
    \end{aligned}
  \end{multline*}
  Apply the BDG inequality to the vector valued martingale
  \begin{equation*}
    t\longmapsto \left( \int_{t\vee i/n}^{t\wedge (i+1)/n}\int_{\R^{+}} \left( \car_{\{r\le r(s)\}}-\car_{\{z\le r'(s)\}} \right)\dif \Poist(s,z),\ i\in \llbracket 0,n-1\rrbracket \right),
  \end{equation*}
  to obtain that 
  \begin{align*}
    \esp{
      \|\Xi_{n}(\overline{R}^{(n)})-\Xi_{n}(\overline{W}^{(n)})\|_{\infty,1}}&\lec{} \frac{1}{\sqrt{n}}
    \esp{\sum_{i=0}^{n-1} \int_{i/n}^{(i+1)/n} |r(s)-r'(s)|\dif s}^{1/2}\\
    &= \frac{1}{\sqrt{n}}\esp{ \int_{0}^{1} |r(s)-r'(s)|\dif s}.
  \end{align*}
  With the particular expression of $r$ and $r'$, we get that 
  \begin{multline*}
    \esp{
      \|\Xi_{n}(\overline{R}^{(n)})-\Xi_{n}(\overline{W}^{(n)})\|_{\infty,1}} \\
    \begin{aligned}
      &\lec{} \frac{1}{\sqrt{n}}\esp{\left| \int_{0}^{n} n\int_{0}^{t}\varphi(nt-ns)(\dif \tilde N^{(n)}(s)-\rho\dif s) \right|}^{1/2}\\
      &= \frac{1}{\sqrt{n}} \esp{\left| n\int_{0}^{1} \Phi(s)\left( \dif \tilde N^{(n)}(s)-\rho\dif s \right) \right|}^{1/2}\\
      &=\esp{\left| \int_{0}^{1} (\tilde N^{(n)}(s)-\rho s) \varphi(s)\dif s \right|}^{1/2},      
    \end{aligned}
  \end{multline*}
  by integration by parts.
  Thus, we have that 
  \begin{equation*}
    \esp{\|\Xi_{n}(\overline{R}^{(n)})-\Xi_{n}(\overline{W}^{(n)})\|_{\infty,1}}\lec{\varphi}\esp{\|\tilde N^{(n)}-\rho .\|_{\infty,1}}^{1/2}.
  \end{equation*}
  In view of \eqref{eq_stability_core:11}, we get
  \begin{equation*}
    \esp{\|\Xi_{n}(\overline{R}^{(n)})-\Xi_{n}(\overline{W}^{(n)})\|_{\infty,1}}\lec{} \frac{1}{\sqrt{n}}\cdotp
  \end{equation*}
  Thus, we can substitute $\overline{W}^{(n)}$ to $\overline{R}^{(n)}$ in
  \eqref{eq_stability_core:13}, and the result follows.
\end{proof}
The convergence of $\overline{N}^{(n)}$ follows from the representation formula
established in \cite{BACRY20132475}. Let
\begin{equation*}
  X^{(n)}(t)=N^{(n)}(t)-\esp{N^{(n)}(t)},\quad t\ge 0.
\end{equation*}
Then, we have for all $t\ge 0$, 
\begin{equation}\label{eq_stability_core:15}
 \overline{ X}^{(n)}(t)=\overline{W}^{(n)}(t)+\int_{0}^{t} n\psi(n s)\, \overline{W}^{(n)}(t-s)\dif s .
\end{equation}
The analysis of this identity is a bit tricky because the two terms of the
integrand do depend on~$n$, so we cannot invoke the Lipschitz continuity of
a well chosen map.
\begin{theorem}
  Assume that there exist $\epsilon\in (0,1/2)$ such that for all $n$, 
  \begin{equation}\label{eq_stability_core:18}
    \int_{n^{\epsilon}}^{\infty} \psi(t)\dif t \lec{} n^{-1/2}.
  \end{equation}
 Then, for all $n\ge 1$ we have that 
  \begin{equation*}
    \KR\left(\Xi_{n}(\overline{ X}^{(n)}),\, \Xi_{n}(B\circ \zeta)\right)\lec{} n^{-1/6}\log(n)
  \end{equation*}
  where
  \begin{equation*}
    \zeta(t)= \frac{\rho}{\sqrt{1-\kappa}}\, t,\quad t \ge 0. 
  \end{equation*}
\end{theorem}
\begin{remark}
  We are here limited to the Kantorovitch-Rubinstein distance because we are
  going to use the convergence of first  moments induced by the  $\KR$-topology, see \eqref{eq_stability_core:14}, which is not valid for
  the other distances investigated above.
\end{remark}
\begin{remark}
  The classical choice of functions $\varphi$ are sums of exponential functions.
  In this situation, the integral of the left-hand-side of
  \eqref{eq_stability_core:18} goes to zero exponentially fast, so that the
  hypothesis is truly satisfied.
\end{remark}
\noindent 
The proof follows closely the lines of the proof of \cite{BACRY20132475}. 
\begin{proof}
Fix $n\ge 1$, and take for granted that 
  \begin{equation}\label{eq_stability_core:16}
    I_{n}:=
    \esp{\sup_{t\in [0,1]}\left| \int_{0}^{t} n\psi(ns) \overline{W}^{(n)}(t-s)\dif s - \frac{\kappa}{1-\kappa} \overline{W}^{(n)}(v)\right|}\lec{} n^{-1/2}.
    \end{equation}
    Then according to \eqref{eq_stability_core:15}, 
    \begin{equation*}
      \KR\left(\Xi_{n}(\overline{ X}^{(n)}),\, \Xi_{n}(B\circ \zeta)\right)\lec{}  \KR\left(\frac{1}{1-\kappa}\,\Xi_{n}(\overline{ W}^{(n)}),\, \Xi_{n}(B\circ \zeta)\right).
    \end{equation*}
    The result then follows from Lemma~\ref{lem:ConvergenceW}.

    We now establish \eqref{eq_stability_core:16}. According to the
    decomposition given in \cite{BACRY20132475}, for any $0<\delta<\eta$, we have
    \begin{align*}
      I_{n}\le \esp{\|\overline{W}^{(n)}\|_{\infty,1}} \int_{\delta n}^{\infty} \psi(t)\dif t +\esp{\alpha_{\eta}(\overline{W}^{(n)})}\int_{0}^{\infty}\psi(t)\dif t.
    \end{align*}
    Since the $\KR$-convergence implies the convergence of first order moments,
    \begin{equation*}
      \sup_{n} \esp{\|\overline{W}^{(n)}\|_{\infty,1}}\lec{} \esp{\|B\circ \gamma\|_{\infty,1}}.
    \end{equation*}
    In view of Theorem~\ref{thm:Lip}, we also have that 
    \begin{equation*}
      \sup_{n} \esp{\alpha_{\eta}(\overline{W}^{(n)})}\lec{\epsilon} \esp{\alpha_{\eta}(B\circ \gamma)}\lec{} \eta^{1/2-\epsilon}
    \end{equation*}
    for any $\epsilon>0$. If we choose $\eta=n^{-1+\epsilon}$ and $\delta=\eta/2$,
    we get
    \begin{equation*}
      I_{n}\lec{} \int_{n^{\epsilon}}^{\infty} \psi(t)\dif t + n^{-1/2-\epsilon/2}. 
    \end{equation*}
    The proof is thus complete.
\end{proof}

\section{Auxiliary result}
\label{sec:proofs}

%
%

%
%


%
%
%

%
%

\begin{proposition}
  \label{prop:momentPoisson}
  Let $(X_i, i=1,\cdots,n)$ be Poisson random variables of parameter $\nu$. The Lambert $W$ function is defined over $[-1/e,\infty]$ by the equation $W(x)e^{W(x)}=x$. Then
  $$\esp{\max_{i=1,\cdots,n}X_i}\leq \frac{\log{n/e^{\nu}}}{W(\log(n/e^{\nu})/\nu e)}=\nu e \exp\left(W\left(\log(n/e^{\nu})/\nu e\right)\right).$$
\end{proposition}
\begin{proof}
  Consider $(Z_i, i=1,\cdots,n)$ some independent centered Poisson variables (i.e., for all $i$, $Z_i=X_i-\nu$). By a straightforward calculation, for all $u\in\R$ and all $i$,
  $$\esp{e^{uZ_i}}=e^{-u\nu}\sum_{k=0}^{\infty}e^{uk}e^{-\nu}\frac{\nu^k}{k!}=e^{-u\nu-\nu}e^{\nu e^u}$$
  Therefore the logarithm of the moment generating function of $Z_i$ is $\Psi_{Z_i}(u)=\nu\left(e^u-u-1\right).$

  By Jensen's inequality, we obtain
  $$\exp\left(u\esp{\max_{i=1,\cdots,n}Z_i}\right)\leq\esp{\exp\left(u\max_{i=1,\cdots,n}Z_i\right)}=\esp{\max_{i=1,\cdots,n}\exp(uZ_i)}$$
  Because the maximum of a sequence of positive numbers is lower than its sum, the right hand side of the last equation is lower than $\displaystyle \esp{\sum_{i=0}^n\exp(uZ_i)}$.\\
  Hence, by the definition of $\Psi_{Z_i}$,
  \begin{align*}
    \exp\left(u\esp{\max_{i=1,\cdots,n}Z_i}\right)&\leq\sum_{i=1}^n\esp{\exp(uZ_1)}\\
                                                  &\leq n\exp\left(\Psi_{Z_i}(u)\right)\\
    &=n\exp(\nu\left(e^u-u-1\right)).
  \end{align*}
  Taking the log, for any $u$ in $\R$,
  $$u\esp{\max_{i=1,\cdots,n}Z_i}-\nu\left(e^u-u-1\right)\leq \log n$$
  so that
  $$\esp{\max_{i=1,\cdots,n}Z_i}\leq \inf_{u\in\R}\left(\frac{\log n+\nu\left(e^u-u-1\right)}{u}\right).$$
By differentiation, it is easy to check that the infimum is reached when
  \begin{equation}
  \nu ue^u-\nu e^u+\nu=\log n.
  \label{infimum}
  \end{equation}
  Therefore, the infimum is equal to
  \begin{equation}
  \frac{\log n+\nu\left(e^{1+W(a)}-1-W(a)-1\right)}{1+W(a)}
  \label{value of infimum}
  \end{equation}
  But we know from \eqref{infimum} that $\nu (1+W(a))e^{1+W(a)}-\log n=\nu e^{1+W(a)}-\nu$ so that \eqref{value of infimum} is equal to
  $$\nu e^{1+W(a)}-\nu=\nu e e{W(a)}-\nu=\nu e \frac{a}{W(a)}-\nu$$
  Remembering that the $Z_i$ are the centered $X_i$ we thus obtain that
  $$\esp{\max_{i=1,\cdots,n}X_i}\leq \nu e \frac{a}{W(a)}-\nu+\nu=\frac{\log{(n/e^{\nu})}}{W(\log(n/e^{\nu})/\nu e)}$$
  which completes the proof.

\end{proof}
  We conclude by observing that $W(z)\geq \log (z) - \log\log(z)$ for all $z >e$. Therefore
  for $n\geq \exp \left(e^{\nu+1}+\nu\right)$,
   using the second expression for the bound of the expectation of the maximum in Proposition \ref{prop:momentPoisson} we get that 
   $$\esp{\max_{i=1,\cdots,n}X_i}\leq \frac{\log(n/e^{\nu})} {\log\left(\frac{\log(n/e^{\nu})}{\nu e}\right)}\cdot$$



\providecommand{\bysame}{\leavevmode\hbox to3em{\hrulefill}\thinspace}
\providecommand{\MR}{\relax\ifhmode\unskip\space\fi MR }
\providecommand{\MRhref}[2]{%
  \href{http://www.ams.org/mathscinet-getitem?mr=#1}{#2}
}
\providecommand{\href}[2]{#2}


\begin{thebibliography}{10}

\bibitem{BACRY20132475}
E.~Bacry, S.~Delattre, M.~Hoffmann, and J.F. Muzy, \emph{Some limit theorems
  for {H}awkes processes and application to financial statistics}, Stochastic
  Processes and their Applications \textbf{123} (2013), no.~7, 2475 -- 2499.

\bibitem{BarbourSteinmethoddiffusion1990}
A.~D. Barbour, \emph{Stein's method for diffusion approximations}, Probability
  Theory and Related Fields \textbf{84} (1990), no.~3, 297--322.

\bibitem{besancon_steins_2018}
E.~Besançon, L.~Decreusefond, and P.~Moyal, \emph{Stein's method for diffusive
  limit of {Markov} processes}, Queueing Systems \textbf{95} (2020), 173--201.

\bibitem{MR4168389}
T.~Bonis, \emph{Stein's method for normal approximation in {W}asserstein
  distances with application to the multivariate central limit theorem},
  Probab. Theory Related Fields \textbf{178} (2020), no.~3-4, 827--860.
  

\bibitem{borovkov1967limit}
AA.~Borovkov, \emph{On limit laws for service processes in multi-channel
  systems}, Siberian Mathematical Journal \textbf{8} (1967), no.~5, 746--763.

\bibitem{Braverman2017}
A. Braverman and J.~G. Dai, \emph{Stein's method for steady-state diffusion
  approximations of {$M/PH/n+M$} systems}, The Annals of Applied Probability
  \textbf{27} (2017), no.~1, 550--581. 

\bibitem{Braverman2016}
A. Braverman, J.~G. Dai, and J. Feng, \emph{Stein's method for
  steady-state diffusion approximations: an introduction through the
  {E}rlang-{A} and {E}rlang-{C} models}, Stochastic Systems \textbf{6} (2016),
  no.~2, 301--366. 

\bibitem{britton}
T.~{Britton} and E.~{Pardoux}, \emph{Stochastic epidemics in a homogeneous
  community}, arXiv e-prints (2018), arXiv:1808.05350.

\bibitem{bremaud_point_1981}
P.~Brémaud, \emph{Point processes and queues, martingale dynamics},
  {Springer-Verlag}, New York, 1981, Springer Series in Statistics.

\bibitem{Chaikin1995}
P.~M. Chaikin and T.~C. Lubensky, \emph{Principles of condensed matter
  physics}, Cambridge University Press.

\bibitem{coutin_steins_2012}
L.~Coutin and L.~Decreusefond, \emph{Stein's method for {Brownian}
  approximations},  (2012) (en).

\bibitem{coutin:hal-02098892}
L.~Coutin and L.~Decreusefond, \emph{Donsker's theorem in {W}asserstein-1
  distance}, Electron. Commun. Probab. \textbf{25} (2020), 1--13.

\bibitem{Coutin:2019aa}
L.~Coutin and L.~Decreusefond, \emph{Stein's method for rough paths}, Potential
  Analysis \textbf{50} (2020), 387--406.

\bibitem{Decreusefondlang}
L. Decreusefond, M. Etienne, G. Lang, S. Robin, and P. Vallois, \emph{{Longest excursion of
  the Ornstein Uhlenbeck Process: applications to genomics and telecom}}, in preparation.

\bibitem{decreusefond2012}
L.~Decreusefond, J.-S. Dhersin, P.~Moyal, and V.~C. Tran, \emph{Large graph
  limit for an {SIR} process in random network with heterogeneous connectivity},
  Ann. Appl. Probab. \textbf{22} (2012), no.~2, 541--575.

\bibitem{decreusefond_functional_2008}
L.~Decreusefond and P.~Moyal, \emph{A functional central limit theorem for the
  {{M}}/{{GI}}/$\infty$ queue}, The Annals of Applied Probability
  \textbf{18} (2008), no.~6, 2156--2178.

\bibitem{decreusefond2012stochastic}
L. Decreusefond and P. Moyal, \emph{Stochastic modeling and analysis
  of telecom networks}, John Wiley \& Sons, 2012.

\bibitem{Etheridge2011}
A.~Etheridge, \emph{Some mathematical models from population genetics}, Lecture
  Notes in Mathematics, vol. 2012, Springer, Heidelberg.

\bibitem{ethier86}
S.N. Ethier and T.G. Kurtz, \emph{{M}arkov processes : Characterizations and
  convergence}, Wiley, 1986.

\bibitem{friz_multidimensional_2010}
P.~K. Friz and N.~B. Victoir, \emph{Multidimensional {Stochastic} {Processes}
  as {Rough} {Paths}: {Theory} and {Applications}}, Cambridge University Press,
  2010.

\bibitem{GW19}
R.~E. Gaunt and N. Walton, \emph{Stein’s method for the single server
  queue in heavy traffic}, Statistics \& Probability Letters \textbf{156}
  (2020), 108566.

\bibitem{Haeusler1984}
E.~Haeusler, \emph{On the rate of convergence in the invariance principle for
  real-valued functions of doeblin processes},  \textbf{15}, 73--90.

\bibitem{Harrison1987}
J.~M. Harrison and R.~J. Williams, \emph{Brownian models of open queueing
  networks with homogeneous customer populations},  \textbf{22}, no.~2,
  77--115.

\bibitem{hawkes74}
A.~G. Hawkes and D.~Oakes, \emph{A cluster process representation of a
  self-exciting process}, Journal of Applied Probability \textbf{11} (1974),
  no.~3, 493--503.

\bibitem{ito_convergence_1968}
K.~It{\^o} and M.~Nisio, \emph{On the convergence of sums of independent
  {{Banach}} space valued random variables}, Osaka Journal of Mathematics
  \textbf{5} (1968), 35--48.

\bibitem{JacodCalculstochastiqueproblemes1979}
J.~Jacod, \emph{Calcul stochastique et probl{\`e}mes de martingales},
  {Springer-Verlag}, 1979.

\bibitem{kasprzak_diffusion_2017}
M.J. Kasprzak, \emph{Diffusion approximations via {{Stein}}'s method and time
  changes}, arXiv:1701.07633 [math] (2017).

\bibitem{Kruk2011}
\L. Kruk, J.~Lehoczky, K.~Ramanan, and S.~Shreve, \emph{Heavy traffic analysis
  for {EDF} queues with reneging},  \textbf{21}, no.~2, 484--545. 

\bibitem{massoulie96}
L.~Massoulié and P.~Brémaud, \emph{Stability of nonlinear hawkes processes},
  The Annals of Probability \textbf{24} (1996), no.~3, 1563--1588.

\bibitem{Raic2018}
M.~Raič, \emph{A multivariate central limit theorem for lipschitz and smooth
  test functions}.

\bibitem{robert_stochastic_2003}
P.~Robert, \emph{Stochastic networks and queues}, french ed., Applications of
  Mathematics, vol.~52, {Springer-Verlag}, Berlin, 2003, Stochastic Modelling
  and Applied Probability.

\bibitem{Sawyer1972}
S.~{Sawyer}, \emph{Rates of convergence for some functionals in probability},
  \textbf{43}, 273--284 (English).

\bibitem{Villani2003}
C.~Villani, \emph{Topics in optimal transportation}, Graduate Studies in
  Mathematics, vol.~58, {American Mathematical Society}, Providence, RI, 2003.

\bibitem{volz2008sir}
E. Volz, \emph{{SIR} dynamics in random networks with heterogeneous
  connectivity}, Journal of mathematical biology \textbf{56} (2008), no.~3,
  293--310.

\end{thebibliography}

\end{document}